\documentclass[a4paper]{amsart}
\usepackage{amsmath}
\usepackage{amssymb}
\usepackage{amsthm,amsxtra}
\usepackage{cite}
\usepackage{enumerate}
\usepackage{enumitem}
\usepackage[mathscr]{eucal}
\usepackage[colorlinks=true]{hyperref}
\hypersetup{hidelinks}
\usepackage{cleveref}
\usepackage{adjustbox}
\usepackage{float}
\usepackage{tikz-cd}
\usepackage{tkz-graph}
\usetikzlibrary{positioning,arrows,shapes.geometric,trees}
\usepackage{graphicx}
\usepackage{mathtools}
\usetikzlibrary{positioning}
\usetikzlibrary{decorations.text}

\usepackage{geometry}
\newgeometry{vmargin={35mm}, hmargin={25mm,25mm}}

\newtheorem{Def}{Definition}[section]
\newtheorem{Th}[Def]{Theorem}
\newtheorem{Ex}[Def]{Example}
\newtheorem{Lemma}[Def]{Lemma}
\newtheorem{Prop}[Def]{Proposition}
\newtheorem{Cor}[Def]{Corollary}

{}

\DeclareMathOperator{\Sf}{S_{fin}}
\DeclareMathOperator{\Gf}{G_{fin}}
\DeclareMathOperator{\S1}{S_1}
\DeclareMathOperator{\G1}{G_1}
\DeclareMathOperator{\Uf}{U_{fin}}

\DeclareMathOperator{\cov}{cov}

\DeclareMathOperator{\CDR}{CDR_{sub}}

\DeclareMathOperator{\bk}{\mathfrak{B}}
\DeclareMathOperator{\bb}{\mathfrak{b}}
\DeclareMathOperator{\df}{\mathfrak{d}}
\DeclareMathOperator{\mf}{\mathfrak{m}}
\DeclareMathOperator{\pf}{\mathfrak{p}}
\DeclareMathOperator{\rf}{\mathfrak{r}}
\DeclareMathOperator{\uc}{\mathcal{U}}
\DeclareMathOperator{\vc}{\mathcal{V}}
\DeclareMathOperator{\ob}{\mathcal{O}_{\mathfrak{B}}}
\DeclareMathOperator{\Gb}{\Gamma_{\mathfrak{B}}}
\DeclareMathOperator{\gb}{\gamma_{\mathfrak{B}}}
\DeclareMathOperator{\ogb}{\mathcal{O}_{g(\mathfrak{B})}}
\DeclareMathOperator{\Ggb}{\Gamma_{g(\mathfrak{B})}}
\DeclareMathOperator{\Split}{Split}

\newcommand{\bfGb}{\boldsymbol{\Gamma}_{\bk}}
\newcommand{\ac}{\mathcal{A}}
\newcommand{\bc}{\mathcal{B}}
\newcommand{\cc}{\mathcal{C}}
\newcommand{\dc}{\mathcal{D}}
\newcommand{\nb}{\mathbb{N}}
\newcommand{\oc}{\mathcal{O}}
\newcommand{\pc}{\mathcal{P}}
\newcommand{\qc}{\mathcal{Q}}
\newcommand{\wc}{\mathcal{W}}
\newcommand{\mc}{\mathcal{M}}
\newcommand{\zc}{\mathcal{Z}}
\newcommand{\fc}{\mathcal{F}}
\newcommand{\gc}{\mathcal{G}}
\newcommand{\ks}{\mathscr{K}}
\begin{document}
\title[Certain results on selection principles in topological spaces]{Certain results on selection principles associated with bornological structure in topological spaces}

\author[D. Chandra, S. Das  and N. Alam ]{ Debraj Chandra$^{\dag}$, Subhankar Das$^{\ddag}$ and Nur Alam$^*$ }
\address{\llap{$\dag$\,}Department of Mathematics, University of Gour Banga, Malda-732103, West Bengal, India}
\email{debrajchandra1986@gmail.com}

\address{\llap{$\ddag$\,}Department of Mathematics, Jadavpur University, Kolkata-700032, West Bengal, India}
\email{subhankarjumh70@gmail.com}

\address{\llap{*\,}Department of Mathematics, Directorate of Open and Distance Learning (DODL), University of Kalyani, Kalyani, Nadia-741235, West Bengal, India}
\email{nurrejwana@gmail.com}


\maketitle

\begin{abstract}
We study selection principles related to bornological covers in a topological space $X$ following the work of Aurichi et al., 2019, where selection principles have been investigated in the function space $C_{\bk}(X)$ endowed with the topology $\tau_{\bk}$ of uniform convergence on bornology $\bk$. We show equivalences among certain selection principles and present some game theoretic observations involving bornological covers. We investigate selection principles on the product space $X^n$ equipped with the product bornolgy $\bk^n$, $n\in \omega$. Considering the cardinal invariants such as the unbounding number ($\bb$), dominating numbers ($\df$), pseudointersection numbers ($\pf$) etc., we establish connections between the cardinality of base of a bornology with certain selection principles. Finally, we investigate some variations of the tightness properties of $C_{\bk}(X)$ and present their characterizations in terms of selective bornological covering properties of $X$.

\end{abstract}

\section{Introduction}
We follow the notation and terminology of \cite{arhan92,mccoy88,Engelking,hh}. The main framework of our work is bornology. Recall that a \textit{bornology} $\bk$ on a space $X$ is a collection of subsets of $X$ which satisfies the following conditions: $(1)$ it is hereditary, i.e., for $B\in \bk$ and $B^\prime \subseteq B$, $B^\prime\in \mathfrak{B}$; $(2)$ it is closed under taking finite unions, i.e., for $B, B^\prime \in \mathfrak{B}$, $B\cup B^\prime \in \mathfrak{B}$, $(3)$ it is a cover of $X$ \cite{hh}. A \textit{base} for bornology $\mathfrak{B}$ on $X$ is a subfamily $\mathfrak{B}_0$ of $\mathfrak{B}$ which is cofinal in $\mathfrak{B}$ with respect to the inclusion, i.e., for each $B \in \mathfrak{B}$ there is $B_0 \in \mathfrak{B}_0$ such that $B\subseteq B_0$. $\bk_0$ is called closed (compact) if each members of $\bk_0$ are closed (compact) subsets of $X$. The family $\fc$ of all finite subsets of $X$, the family $\pc(X)$ of all subsets of $X$ and the family $\ks$ of nonempty subsets of $X$ with compact closure are few examples of bornologies on $X$. The bornology $\fc$ is the smallest while $\pc(X)$ is the largest bornology on $X$.

A systematic study of selection principles was initiated by Scheepers in his seminal papers \cite{coc1,coc2}. Since then a lot of researches have been carried out in this field of mathematics resulting a vast body of literatures. The theory of selection principles has connections with several disciplines of mathematics such as Set theory, Topological groups, General topology, Ramsey theory, Cardinal invariants etc. Following the pioneering work of Scheepers many researchers have explored various generalizations and weaker forms of certain notions (for examples, $\gamma$-cover, large cover, the Hurewicz property etc.) of selection principles by employing diverse frameworks such as ideals, bornologies, compact sets, bounded sets etc. Moreover, the study of selection principles in function spaces is one of the important line of investigation where several topological properties of the function space $C(X)$ have been characterized in terms of selective covering properties of $X$. Primarily, the function space endowed with the topology of pointwise convergence $(C(X),\tau_p)$( in short, $C_p(X)$) and the function space endowed with the compact open topology $(C(X),\tau_k)$ (in short, $C_k(X)$) are two prominent topological spaces where extensive investigation have been carried out resulting in nice characterizations of well known properties of selection principles (for examples, the Hurewicz, tightness, Fr\'{e}chet-Urysohn properties etc.). One can see survey papers \cite{sur1,sur2,sur3, sur4, sur5} and references therein for detailed information on selection principles.

As far as bornological investigation of selection principles is concerned, it has been started by Caserta, Di Maio, Ko\v{c}inac in \cite{caserta12} where selection principles have been studied in the function space $C(X)$ endowed with the topology of strong uniform convergence  on bornology \cite{bl}. Subsequently, in \cite{aurichi16, am} the authors taking a broader perpective studied selection principles in the function space $C_{\bk}(X)$ using the topology $\tau_{\bk}$ of uniform convergence on $\bk$. They have investigated certain dual games between a Tychonoff space $X$ and $C_{\bk}(X)$, certain selective local properties of $C_{\bk}(X)$ corresponding to some bornological covering properties of $X$ and also investigated $\gb$-spaces while assuming a bornology with a compact base throughout. This work has been further continued in \cite{Fernandez} with a focus on certain topological games with respect to class of dense subsets of $C_{\bk}(X)$. We intend to follow this line of investigation and present several results on selection principles on $X$ and on certain local properties of  $C_{\bk}(X)$.

The paper is arranged as follows. In section $2$, definitions of basic notions and properties of $X$ which are used subsequently are presented. In section $3$, we present equivalences of certain selection principles, some game theoretic observations and discuss the $\alpha_i$-properties of $X$ in relation to bornological covers. In section $4$, we focus on selection principles on the product space $X^n$ with respect to the product bornology $\bk^n$, $n\in \omega$, and investigate behaviour of some selection principles under a continuous image. We also establish connections between cardinalities of base $\bk_0$ of bornology $\bk$ on $X$ with certain selection principles on $X$. The cardinal invariants that we have considered are unbounding number ($\bb$), dominating numbers ($\df$), pseudointersection numbers ($\pf$),  $\cov(\mc)$ etc. Finally, section $5$ is focused on the function space $C_{\bk}(X)$ endowed with the topology $\tau_{\bk}$ of uniform convergence $\bk$. We show equivalences of various local properties of $C_{\bk}(X)$ in particular certain variations of the tightness property of $C_{\bk}(X)$ with selective bornological covering properties of $X$. We prove that the supertightness property and the tightness property of $C_{\bk}(X)$ coincide and also establish relation between the Lindel\"{o}f number of $C_{\bk}(X)$ with the supertightness of $X^n$, $n\in \omega$.

\section{Preliminaries}
Throughout $X$ stands for a topological space.  Let $\mathcal{A}$ and $\mathcal{B}$ be collections consisting of families of subsets of $X$.
\begin{enumerate}[leftmargin=2cm]
\item[$\S1(\mathcal{A},\mathcal{B})$:] For each sequence $\{\mathcal{U}_n:n\in \omega\}$ of elements of $\mathcal{A}$ there exists a sequence $\{V_n:n\in \omega\}$ such that for each $n$ $V_n\in\mathcal{U}_n$ and $\{V_n : n\in\omega\}\in\mathcal{B}$ \cite{coc1,coc2}.

\item[$\Sf(\mathcal{A},\mathcal{B})$:] For each sequence $\{\mathcal{U}_n:n\in \omega\}$ of elements of $\mathcal{A}$ there exists a sequence $\{\vc_n:n\in \omega\}$ such that for each $n$ $\mathcal{V}_n$ is a finite (possibly empty) subset of $\mathcal{U}_n$ and $\bigcup_{n\in\omega}\mathcal{V}_n\in\mathcal{B}$ \cite{coc1,coc2}.

\item[$\Uf(\mathcal{A},\mathcal{B})$:] For each sequence $\{\mathcal{U}_n:n\in \omega\}$ of elements of $\mathcal{A}$ there exists a sequence $\{\vc_n:n\in \omega\}$ such that for each $n$ $\mathcal{V}_n$ is a finite (possibly empty) subset of $\mathcal{U}_n$ and $\{\bigcup\mathcal{V}_n : n\in\omega\}\in\mathcal{B}$ or $\bigcup\mathcal{V}_n=X$ for some $n$ \cite{coc1,coc2}.

\item[$\bigcap_{\infty}(\ac,\bc)$:] For each sequence $\{\uc_n:n\in \omega\}$ of elements of $\ac$
there exists a sequence $\{\vc_n:n\in \omega\}$ such that $\vc_n$ is an infinite subset of $\uc_n$ for each $n\in \omega$ and $\{\cap \vc_n:n\in \nb\}\in \bc$ \cite{tb2}.

\item[$\CDR(\ac,\bc)$:] For each sequence $\{\uc_n:n\in \omega\}$ of elements of $\ac$ there exists a sequence $\{\vc_n:n\in \nb\}$ such that for each $n$, $\vc_n\subseteq \uc_n$, for $m\neq n$, $\vc_m\cap \vc_n=\emptyset$ and each $\vc_n$ is a member of $\bc$\cite{coc1}.

\item[$\Split(\ac, \bc)$:] For each element $\uc$ of $\ac$ there exist two elements $\vc_1$ and $\vc_2$ of $\bc$ such that $\vc_1\cap \vc_2=\emptyset$ and $\vc_1,\vc_2\subseteq \uc$ \cite{coc1}.

\end{enumerate}

Corresponding to the $\S1(\ac,\bc)$ and $\Sf(\ac,\bc)$ selection principles the associated infinitely long games are given below.
\begin{enumerate}[leftmargin=2cm]
\item[$\G1(\ac,\bc)$:] Two players, ONE and TWO, play a round for each $n\in \omega$. In the $n$-th round ONE chooses a set $\uc_n$ from $\ac$ and TWO responds by choosing an element $V_n\in  \uc_n$. TWO wins the play $\{\uc_1,V_1, \dotsc, \uc_n, V_n, \dotsc \}$ if $\{V_n :n\in \omega\}\in \bc$. Otherwise ONE wins \cite{coc1,coc2}.

\item[$\Gf(\ac,\bc)$:] Two players, ONE and TWO, play a round for each $n\in \omega$. In the $n$-th round ONE chooses a set $\uc_n$ from $\ac$ and TWO responds by choosing a finite (possibly empty) set $\vc_n\subseteq \uc_n$. TWO wins the play $\{\uc_1,\vc_1, \dotsc, \uc_n,\vc_n, \dotsc  \}$ if $\bigcup_{n\in \omega}\vc_n\in \bc$. Otherwise ONE wins \cite{coc1,coc2}.
\end{enumerate}

In \cite{alpha} the $\alpha_i$ properties are defined as follows. The symbol $\alpha_i(\ac,\bc)$ for $i=1,2,3,4$ denotes that for each sequence $\{\uc_n:n\in \omega\}$ of elements of $\ac$, there exists $\vc\in \bc$ such that
\begin{enumerate}[leftmargin=2cm]
\item[$\alpha_1(\ac,\bc)$:] For each $n\in \omega$, the set $\uc_n\setminus \vc$ is finite.

\item[$\alpha_2(\ac,\bc)$:] For each $n\in \omega$, the set $\uc_n\cap \vc$ is infinite.

\item[$\alpha_3(\ac,\bc)$:] For infinitely many $n\in\omega$, the set $\uc_n\cap \vc$ is infinite.

\item[$\alpha_4(\ac,\bc)$:] For infinitely many $n\in \omega$, the set $\uc_n\cap \vc$ is nonempty.
\end{enumerate}

We now recall some cardinal invariants which will be used in our investigation (see \cite{Vaughan} for more details). A natural pre-order $\leq^*$ on the Baire space $\omega^\omega$ is defined by $f\leq^*g$ if and only if $f(n)\leq g(n)$ for all but finitely many $n$. A subset $A$ of $\omega^\omega$ is said to be bounded if there is a $g\in\omega^\omega$ such that $f\leq^*g$ for all $f\in A$. Let  $\mathfrak{b}$ denote the smallest cardinality of an unbounded subset of $\omega^\omega$. A subset $D$ of $\omega^\omega$ is dominating if for each $g\in\omega^\omega$ there exists a $f\in D$ such that $g\leq^* f$. Let $\mathfrak{d}$ denote the minimum cardinality of a dominating subset of $\omega^\omega$ and $\mathfrak{c}$ be the cardinality of the set of reals. The value of $\mathfrak{d}$ does not change if one considers the relation `$\leq$' instead of `$\leq^*$' \cite{KKJE}. Suppose that $\cc$ is a family of infinite subsets of $\omega$. By $P(\cc)$ we denote that there exists a subset $P$ of $\omega$ for which for each $C\in \cc$, the set $P\setminus C$ is finite. Let $\pf$ denote the smallest cardinal number $k$ such that the following statement is false: For each family $\cc$ if any finite subfamily of $\cc$ has infinite intersection and $|\cc|\leq k$, then $P(\cc)$ holds. Let $\cov(\mc)$  denote the smallest cardinal number $k$ such that a family of $k$ first category subsets of the real line covers the real line. Let $\cc$ be a family of subsets of $\omega^\omega$ with $|\cc|<\cov(\mc)$. Then there is an $h\in \omega^\omega$ for which  the set $\{n\in \omega:f(n)=h(n)\}$ is infinite for every $f\in \cc$ (see \cite[Theorem 5]{fm}).
The relations among these cardinal numbers are as follows. $\pf\leq \bb\leq \df$, $\pf\leq \cov(\mc)$.
Let $[\omega]^\omega$ denote the collection of all infinite subsets of $\omega$. A set $R\in[\omega]^\omega$ reaps a family $\mathcal{A} \subseteq [\omega]^\omega$ if for each $P\in\mathcal{A}$, both the sets $P\cap R$ and $P\setminus R$ are infinite. Let $\mathfrak{r}$ be the minimum cardinality of a family $\mathcal{A}\subseteq[\omega]^\omega$ that no infinite subset $R$ of $\omega$ reaps.

Let $\cc$ be an infinite family of subsets of $X$. $\cc$ \textit{converges to} $x$ if for each open set $U$ containing $x$, the set $\{C\in \cc:C\nsubseteq U\}$ is finite.
Let $x\in X$. By $\Omega_x$ we denote the family $\{C\subseteq X:x\in \overline{C}\setminus C\}$ \cite{cocVII}, and by $\Sigma_x$ we denote the family of all sequences converging to $x\in X$ \cite{kock2}. The symbol $t(X)$ denotes the \textit{tightness } of $X$ which is the smallest infinite cardinal $\mf$ such that when $C\subseteq X$ and $x\in \overline{C}$, there exists $D\subseteq C$ with $x\in \overline{D}$ and $|D|\leq \mf$ \cite{arh76, j}. If $t(X)=\omega$, then $X$ has \textit{countable tightness}.  $X$ has \textit{countable fan tightness} (\textit{countable strong fan tightness}) at $x$ if $X$ satisfies $\Sf(\Omega_x, \Omega_x)$ ($\S1(\Omega_x, \Omega_x)$) \cite{arhan92, sakai88}. Let $\cc$ be a family of nonempty subsets of $X$. $\cc$ is a \textit{$\pi$-network} at $x\in X$ if each neighbourhood of $x$ contains some member of $\cc$. $X$ has \textit{countable fan tightness for finite sets}  if for each $x$ in $X$ and for each sequence $\{\cc_n:n\in \omega\}$ of $\pi$-networks at $x$ which consist of finite subsets of $X$, there exists a finite subset $\dc_n$ of $\cc_n$ for each $n$ for which $\bigcup_{n\in \omega}\dc_n$ is a $\pi$-network at $x$ \cite{sakai2}. $X$ has \textit{countable strong fan tightness for finite sets} if $x$ in $X$ and for each sequence $\{\cc_n:n\in \omega\}$ of $\pi$-networks at $x$ which consist of finite subsets of $X$, there exists an $C_n\in \cc_n$ for each $n$  for which $\{C_n:n\in \omega\}$ is a $\pi$-network at $x$ \cite{sakai2}. The \textit{supertightness} $st(x,X)$ of $x$, $x\in X$, is the smallest cardinal $\mf$ for which each $\pi$-network $\cc$ at $x$ consisting of finite subsets of $X$ has a subcollection $\dc$ with $|\dc|\leq \mf$ which is a $\pi$-network at $x$. The supertightness of $X$ is denoted by $st(X)$ and  is defined by $st(X)=\omega\cdot \sup\{st(x,X):x\in X\}$ \cite{mm} (see also \cite{sakai-st}). If $st(X)=\omega$, then $X$ has \textit{countable supertightness}. $X$ is \textit{Fr\'{e}chet–Urysohn} if for each  subset $C$ of $X$ and each $x\in \overline{C}$, there exists a sequence in $C$ which converges to $x$ \cite{gn}. $X$ is \textit{strictly Fr\'{e}chet-Urysohn} if $X$ satisfies $\S1(\Omega_x, \Sigma_x)$ for each $x\in X$ \cite{gn}. $X$ is \textit{Fr\'{e}chet-Urysohn for finite sets} if for each $x\in X$ and each $\pi$-network $\cc$ at $x$ which consist of finite subsets of $X$, $\cc$ contains a subfamily which converges to $x$ \cite{gs, rs}. Observe that the "Fr\'{e}chet-Urysohn for finite sets" implies the "Fr\'{e}chet–Urysohn". An equivalent formulation of the property Fr\'{e}chet-Urysohn for finite sets is that for each $x$ in $X$ and each sequence $\{\cc_n:n\in \omega\}$ of $\pi$-networks at $x$ which consist of finite subsets of $X$ there exists $C_n\in \cc_n$ for each $n\in \omega$ such that $\{C_n:n\in \omega\}$ converges to $x$ (see \cite[Proposition 6]{rs}).

We now recall various classes of open covers of $X$. Let $\mathcal O$ denote the collection of all open covers of $X$.
An open cover $\mathcal{U}$ of $X$ is an \textit{$\omega$-cover} if $X\notin\mathcal{U}$ and for each finite subset $F$ of $X$ there exists a $U\in\mathcal{U}$ such that $F\subseteq U$ \cite{coc1,coc2}. The symbol $\Omega$ denotes the collection of all $\omega$-covers of $X$. An open cover $\mathcal{U}$ of $X$ is a \textit{$\gamma$-cover} if $\mathcal{U}$ is infinite and for each $x\in X$ the set $\{U \in \mathcal{U} : x\notin U\}$ is finite \cite{coc2,coc1}. The symbol $\Gamma$ denotes the collection of all $\gamma$-covers of $X$. The \textit{Lindel\"{o}f number} $l(X)$ of $X$ is the smallest infinite cardinal $\mf$ such that each open cover of $X$ has a subcover with cardinality less than or equal to $\mf$ \cite{arh76, j}. Let $k\in \omega$. A cover $\mathcal{U}$ of $X$ is $k$-cover \cite{tsaban24} if $X\notin \mathcal{U}$ and for each $k$-element subset $F$ of $X$ there exists $U\in\mathcal{U}$ such that $F\subseteq U$. Observe that a cover is an $\omega$-cover if and only if it is an open $k$-cover for each $k\in \omega$. By $\mathcal{O}_k$ we denote the collection of all open $k$-covers of $X$. Let $\bk$ be a bornology on $X$ with a closed base. A collection $\mathcal{U}$ of subsets of $X$ is \textit{$\mathfrak{B}$-cover} for $X$ if $X\notin \mathcal{U}$ and for every $B\in \mathfrak{B}$ there exists $U \in \mathcal{U}$ such that $B \subseteq U$. The symbol $\mathcal{O}_\mathfrak{B}$ denotes the collection of all open $\mathfrak{B}$-covers for $X$ \cite{am}. Any cofinite subset of a $\mathfrak{B}$-cover is also a $\mathfrak{B}$-cover of $X$. We call a cover $\uc=\{U_n:n\in \omega\}$ of $X$ a $\bk$-sequence if it is infinite and for each $B\in \bk$ there exists $n_0\in \omega$ such that $B\subseteq U_n$ for all $n\geq n_0$. Let $\bfGb$ denote the collection of all such covers of $X$. If the members of $\uc$ are proper open subsets of $X$, then $\uc$ is a $\gb$-cover of $X$ \cite{am}. The collection of all $\gb$-covers of $X$ is denoted by $\Gb$. The following relations among the classes of covers hold. $\ob\subseteq \Omega\subseteq \oc$, $\Gb\subseteq \ob$ and $\Gb\subseteq \Gamma$. $X$ is a \textit{$\bk$-Lindel\"{o}f space} if each open $\mathfrak{B}$-cover of $X$ contains a countable $\bk$-subcover.

For any two spaces $X$ and $Y$, $Y^X$ ($C(X, Y)$) stands for the set of all functions (continuous functions) from $X$ to $Y$. $C(X)$ denotes the set of all continuous real-valued functions on $X$. The commonly used topologies on $C(X, Y)$ are the compact-open topology $\tau_k$, and the topology of pointwise convergence $\tau_p$. The corresponding spaces are, in general, denoted by $(C(X, Y), \tau_k)$ ($C_k(X)$ when $Y = \mathbb{R}$), and $(C(X, Y), \tau_p)$ ($C_p(X)$ when $Y = \mathbb{R}$).

The topology of uniform convergence on $\mathfrak{B}$, denoted by $\tau_\mathfrak{B}$, is the topology on $C(X)$ that has a neighborhood base at each $f\in C(X)$ the sets of the form
\[\langle B, \epsilon \rangle[f] = \{g\in C(X) : \forall x\in B (|f(x) - g(x)|< \epsilon)\},\] for $B\in \mathfrak{B}$ and $\epsilon > 0$. By $C_\mathfrak{B}(X)$ we mean the space $(C(X), \tau_\mathfrak{B})$. It can be shown that $\tau_\mathfrak{B}$ is obtained from a separating uniformity over $C(X)$, from which it follows that $C_\mathfrak{B}(X)$ is a Tychonoff space (see \cite{mccoy88}). It is also worthwhile to mention that $C_\mathfrak{B}(X)$ is a homogeneous space, so there is no loss of generality in fixing an appropriate point from $C_\mathfrak{B}(X)$ in order to analyze its closure properties - in this case, we fix the zero function $\underline{0} : X \to \mathbb{R}$. Observe that if $\bk=\fc$, then $C_\fc(X) = C_p(X)$ and the $\fc$-covers turns out to be the $\omega$-covers of $X$. Also, if $\bk=\ks$, then it follows that $C_\ks(X) = C_k(X)$. One readily sees that $\mathcal{O}_\ks = \mathcal{K}$, where $\mathcal{K}$ denotes the collection of the so called $k$-covers of $X$.

Let $g:X\rightarrow Y$ be a function and $\bk$ be a bornology on $X$.  The family $g(\bk)=\{g(B):B\in \bk\}$ is a bornology on $g(X)$ and whenever $g$ is surjective $g(\bk)$ is a bornology on $Y$. When $g:X\rightarrow Y$ is a continuous function and $\bk$ is a bornology on $X$ with a compact base $\bk_0$, then $g(\bk)$ is a bornology on $g(X)$ with compact base $g(\bk_0)$. For $n\in \omega$ by $\bk^n$ we denote the bornology on $X^n$ generated by the set $\{B^n:B\in \bk\}$ \cite{hh}. Throughout we follow the convention that if $\bk$ is a bornology on $X$, then $X\not\in \bk$.

\section{Certain observations on selection principles and bornological covers}
\subsection{Implications among certain selection principles}
We first present some basic observations on bornological covers which we will use frequently.
\begin{Lemma}
\label{lemma1}
Let $\bk$ be a bornology on a topological space $X$ with a closed base. The following assertions are equivalent.
\begin{enumerate}[wide=0pt, label={\upshape(\arabic*)}, leftmargin=*]
  \item An infinite collection $\mathcal{U}$ of open subsets of $X$ belongs to $\mathcal{O}_\mathfrak{B}$.
  \item For each $B\in \mathfrak{B}$ the set $\{U\in \mathcal{U} : B\subseteq U\}$ is infinite.
\end{enumerate}
\end{Lemma}
\begin{proof}
We prove only $(1) \Rightarrow (2)$. Let $B \in \mathfrak{B}$. If possible suppose that the set $\{U\in \mathcal{U} : B\subseteq U\}$ is finite. Choose $\{U\in \mathcal{U} : B\subseteq U\} = \{U_i : i\in n\}$. Pick a $x \in X$ such that $x \notin U_i$ for all $i \in n$. Since $B\cup \{x\} \in \mathfrak{B}$, there exists a $U \in \mathcal{U}$ such that $B\cup \{x\} \subseteq U$. Which is absurd. Thus $(2)$ holds.
\end{proof}

\begin{Lemma}
\label{lemma3}
Let $\bk$ be a bornology on a topological space $X$ with a closed base. If $\{\mathcal{U}_n:n\in \omega\}$ is a sequence of $\mathfrak{B}$-covers of $X$ such that for each $n$ $\mathcal{U}_n = \{U_m^n : m\in \omega\}$, then for each $n$ $\mathcal{V}_n = \{\cap_{i\in n+1} U_{m_i}^i : m_i\in \omega\}$ is a $\mathfrak{B}$-cover of $X$.
\end{Lemma}
\begin{proof}
Let $B\in \mathfrak{B}$. Then for each $i\in n+1$ we get a $U_{m_i}^i \in \mathcal{U}_i$ such that $B \subseteq U_{m_i}^i$. It follows that $B \subseteq \cap_{i\in n+1} U_{m_i}^i$. Since $\cap_{i\in n+1} U_{m_i}^i \in \mathcal{V}_n$, $\mathcal{V}_n$ is a $\mathfrak{B}$-cover of $X$.
\end{proof}

\begin{Lemma}
\label{lemma2}
Let $\bk$ be a bornology on a topological space $X$ with a closed base. The following assertions hold.
\begin{enumerate}[wide=0pt, label={\upshape(\arabic*)}, leftmargin=*, ref={\theLemma(\arabic*)}]
  \item\label{lemma201} Every infinite subset of a $\gamma_\mathfrak{B}$-cover of $X$ is a $\gamma_\mathfrak{B}$-cover of $X$.
  \item\label{lemma202} If $\{\mathcal{U}_n:n\in \omega\}$ is a sequence of $\gamma_\mathfrak{B}$-covers of $X$ such that for each $n$ $\mathcal{U}_n = \{U_m^n : m\in \omega\}$, then for each $n$ $\mathcal{V}_n = \{\cap_{i\in n+1} U_m^i : m\in \omega\}$ is a $\gamma_\mathfrak{B}$-cover of $X$.
\end{enumerate}
\end{Lemma}
\begin{proof}
$(1)$. Let $\mathcal{U}$ be a $\gamma_\mathfrak{B}$-cover of $X$. Choose an infinite subset $\mathcal{V}$ of $\mathcal{U}$. It is sufficient to show that $\mathcal{V} \in \mathcal{O}_\mathfrak{B}$. Let $B \in \mathfrak{B}$. Since $\{U\in \mathcal{U} : B\nsubseteq U\}$ is finite and $\mathcal{V}$ is infinite, there exists a $V \in \mathcal{V}$ such that $B \subseteq V$. Thus $\mathcal{V} \in \mathcal{O}_\mathfrak{B}$.

$(2)$. Let $B \in \mathfrak{B}$. Then for each $i\in n+1$ there exists $m_i \in \omega$ such that $B \subseteq U_{m_i}^i$ for all $m \geq m_i$. Choose $k = \max \{m_i : i\in n+1\}$. Clearly for each $i\in n+1$, $B \subseteq U_k^i$ for all $m \geq k$. Thus $B \subseteq \cap_{i\in n+1} U_k^i$ with $\cap_{i\in n+1} U_k^i \in \mathcal{V}_n$. Hence $\mathcal{V}_n$ is a $\gamma_\mathfrak{B}$-cover of $X$.
\end{proof}

\begin{Lemma}
\label{lemma4}
Let $\bk$ be a bornology on a topological space $X$ with a closed base. If $\mathcal{U} = \{U_n : n\in \omega\}$ is an open $\mathfrak{B}$-cover of $X$, then $\mathcal{V} = \{\cup_{i\in n+1} U_i : n\in \omega\}$ is a $\gamma_\mathfrak{B}$-cover of $X$.
\end{Lemma}
\begin{proof}
Let $B \in \mathfrak{B}$. Then there exists a $k \in \omega$ such that $B \subseteq U_k$. Since $U_k \subseteq \cup_{i\in n+1} U_i$ for all $n \geq k$, $B \subseteq \cup_{i\in n+1} U_i$ for all $n \geq k$. Thus $\mathcal{V}$ is a $\gamma_\mathfrak{B}$-cover of $X$.
\end{proof}

We introduce the following.
\begin{Def}
Let $\bk$ be a bornology on a topological space $X$ with a closed base. An open cover $\uc$ of $X$ is called $\bk$-groupable if $\uc=\cup_{n\in \omega}\uc_n$, where $\uc_n$'s are pairwise disjoint finite subsets of $\uc$, and for any $B\in \bk$ there exists $n_0$ for which $B\subseteq U$ for some $U\in \uc_n$ for all $n\geq n_0$.
\end{Def}
We denote the collection of all open $\bk$-groupable covers by $\mathcal{O}^{gp}_{\bk}$.

\begin{Ex}
Let $X=\omega^\omega$ and let $\bk$ be a bornology on $X$ with a compact base. Let $\uc=\{U_k:k\in \omega\}$, where $U_k=\{f\in \omega^\omega:f(1)\leq k\}$. Clearly, $U_k\subseteq U_{k+1}$. It is easy to see that $\uc$ is an open $\bk$-cover of $X$. Choose a sequence $n_1<n_2<\cdots$. Take $\uc_1=\{U_k:k\leq n_1\}$ and $\uc_m=\{U_k:n_{m-1}<k\leq n_m\}$, $m>1$. $\uc_m$'s are pairwise disjoint and finite. Also for a $B\in \bk$ we can find $m_0$ such that $B\subseteq U_k$ for $U_k\in \uc_m$ for all $m\geq m_0$. Hence $\uc$ is $\bk$-groupable.
\end{Ex}

\begin{Th}
\label{thm1}
Let $\bk$ be a bornology on a topological space $X$ with a closed base. The following assertions hold.
\begin{enumerate}[wide=0pt, label={\upshape(\arabic*)}, leftmargin=*, ref={\theTh(\arabic*)}]
  \item\label{thm101} $\S1(\Gamma_\mathfrak{B}, \Gamma_\mathfrak{B}) = \Sf(\Gamma_\mathfrak{B}, \Gamma_\mathfrak{B})$.

  \item\label{thm102} $\S1(\Gamma_\mathfrak{B}, \Gamma) = \Sf(\Gamma_\mathfrak{B}, \Gamma)$.

  \item\label{thm103} $\S1(\Gamma, \Gamma_\mathfrak{B}) = \Sf(\Gamma, \Gamma_\mathfrak{B})$.
\end{enumerate}
\end{Th}
\begin{proof}
We provide a proof only for (2); the remaining assertions can be proved similarly.
Obviously, $\S1(\Gamma_\mathfrak{B}, \Gamma)$ implies $\Sf(\Gamma_\mathfrak{B}, \Gamma)$. Conversely, let $X$ satisfy $\Sf(\Gamma_\mathfrak{B}, \Gamma)$. Let $(\mathcal{U}_n)$ be a sequence of $\gamma_\mathfrak{B}$-covers of $X$. Without loss generality we assume that for each $n$ $\mathcal{U}_n = \{U_m^n : m\in \omega\}$. We now define a sequence $(\mathcal{W}_n)$ of $\gamma_\mathfrak{B}$-covers of $X$ by $\mathcal{W}_n = \{V_m^n : m\in \omega\}$ with $V_m^n = \cap_{i\in n+1} U_m^i$. Since $X$ satisfies $\Sf(\Gamma_\mathfrak{B}, \Gamma)$, there exists a sequence $(\mathcal{H}_n)$ such that for each $n$ $\mathcal{H}_n$ is a finite subset of $\mathcal{W}_n$ and $\cup_{n\in \omega} \mathcal{H}_n$ is a $\gamma$-cover of $X$. We can find a sequence $n_0 < n_1 < n_2 < \cdots$ of members of $\omega$ such that for each $k$ $\mathcal{H}_{n_k} \setminus \cup_{m\in k} \mathcal{H}_{n_m} \neq \emptyset$ since $\cup_{n\in \omega} \mathcal{H}_n$ being a $\gamma$-cover of $X$ is infinite. For each $k$ pick a $m_k \in \omega$ such that $V_{m_k}^{n_k} \in \mathcal{H}_{n_k} \setminus \cup_{m\in k} \mathcal{H}_{n_m}$. Clearly $\{V_{m_k}^{n_k} : k\in \omega\}$ is a $\gamma$-cover of $X$. Define a sequence $(U_n)$ such that for each $n$ $U_n \in \mathcal{U}_n$ and given by $U_n = U_{m_{k+1}}^n$ for each $n$ with $n_k < n \leq n_{k+1}$. Observe that $\{U_n : n\in \omega\}$ is a $\gamma$-cover of $X$. Thus $X$ satisfies $\S1(\Gamma_\mathfrak{B}, \Gamma)$.
\end{proof}

\begin{Th}
\label{thm5}
Let $\bk$ be a bornology on a topological space $X$ with a closed base. The following assertions hold.
\begin{enumerate}[wide=0pt, label={\upshape(\arabic*)}, leftmargin=*, ref={\theTh(\arabic*)}]
  \item\label{thm501} $\Uf(\mathcal{O}, \mathcal{O}_\mathfrak{B}) = \Uf(\Omega, \mathcal{O}_\mathfrak{B})$.

  \item\label{thm502} $\Uf(\mathcal{O}, \Gamma_\mathfrak{B}) = \Uf(\Omega, \Gamma_\mathfrak{B})$.
\end{enumerate}
\end{Th}
\begin{proof}
$(1)$. Clearly, $\Uf(\mathcal{O}, \mathcal{O}_\mathfrak{B})$ implies $\Uf(\Omega, \mathcal{O}_\mathfrak{B})$. Conversely, suppose that $X$ satisfies $\Uf(\Omega, \mathcal{O}_\mathfrak{B})$. Let $(\mathcal{U}_n)$ be a sequence of open covers of $X$. We now define a sequence $(\mathcal{W}_n)$ of $\omega$-covers of $X$ as follows. For each $n$ let $\mathcal{W}_n$ be the collection of all finite unions of members of $\mathcal{U}_n$. Since $X$ satisfies $\Uf(\Omega, \mathcal{O}_\mathfrak{B})$, there exists a sequence $(\mathcal{H}_n)$ such that for each $n$ $\mathcal{H}_n$ is a finite subset of $\mathcal{W}_n$ and $\{\cup \mathcal{H}_n : n\in \omega\}$ is an open $\mathfrak{B}$-cover of $X$. Now $(\mathcal{H}_n)$ gives a sequence $(\mathcal{V}_n)$ such that for each $n$ $\mathcal{V}_n$ is a finite subset of $\mathcal{U}_n$ and $\cup \mathcal{V}_n = \cup \mathcal{H}_n$. Thus $X$ satisfies $\Uf(\mathcal{O}, \mathcal{O}_\mathfrak{B})$.
The proof of $(2)$ is similar.
\end{proof}

\begin{Th}
\label{thm6}
Let $\bk$ be a bornology on a topological space $X$ with a closed base. Let $X$ be a Lindel\"{o}f space. The following assertions are equivalent.
\begin{enumerate}[wide=0pt, label={\upshape(\arabic*)}, leftmargin=*, ref={\theTh(\arabic*)}]
  \item\label{thm601} $\Uf(\mathcal{O}, \mathcal{O}_\mathfrak{B}) = \Uf(\Gamma, \mathcal{O}_\mathfrak{B})$.

  \item\label{thm602} $\Uf(\mathcal{O}, \Gamma_\mathfrak{B}) = \Uf(\Gamma, \Gamma_\mathfrak{B})$.
\end{enumerate}
\end{Th}
\begin{proof}
$(1)$. Clearly, $\Uf(\mathcal{O}, \mathcal{O}_\mathfrak{B})$ implies $\Uf(\Gamma, \mathcal{O}_\mathfrak{B})$. Conversely, suppose that $X$ satisfies $\Uf(\Gamma, \mathcal{O}_\mathfrak{B})$. Let $(\mathcal{U}_n)$ be a sequence of open covers of $X$. Without loss generality we assume that for each $n$ $\mathcal{U}_n = \{U_m^n : m\in \omega\}$. We now define a sequence $(\mathcal{W}_n)$ of $\gamma$-covers of $X$ by $\mathcal{W}_n = \{V_m^n : m\in \omega\}$ with $V_m^n = \cup_{i\in m+1} U_i^n$. Since $X$ satisfies $\Uf(\Gamma, \mathcal{O}_\mathfrak{B})$, we get a sequence $(m_n : n\in \omega)$ of elements of $\omega$ such that $\{V_{m_n}^n : n\in \omega\}$ is an open $\mathfrak{B}$-cover of $X$. For each $n$ choose $\mathcal{V}_n = \{U_i^n : i\in m_n+1\}$. Then for each $n$ $\mathcal{V}_n$ is a finite subset of $\mathcal{U}_n$ and $\{\cup \mathcal{V}_n : n\in \omega\}$ is an open $\mathfrak{B}$-cover of $X$. Thus $X$ satisfies $\Uf(\mathcal{O}, \mathcal{O}_\mathfrak{B})$. The proof of $(2)$ is analogous.
\end{proof}

\begin{Th}
\label{thm2}
Let $\bk$ be a bornology on a topological space $X$ with a closed base. Let $X$ be a $\mathfrak{B}$-Lindel\"{o}f space. The following assertions hold.
\begin{enumerate}[wide=0pt, label={\upshape(\arabic*)}, leftmargin=*, ref={\theTh(\arabic*)}]
  \item\label{thm201} $\S1(\mathcal{O}_\mathfrak{B}, \Gamma_\mathfrak{B}) = \Sf(\mathcal{O}_\mathfrak{B}, \Gamma_\mathfrak{B})$.

  \item\label{thm202} $\S1(\mathcal{O}_\mathfrak{B}, \Gamma) = \Sf(\mathcal{O}_\mathfrak{B}, \Gamma)$.
\end{enumerate}
\end{Th}
\begin{proof}
We provide a proof of (2) and omit (1), as the latter is analogous. Obviously, $\S1(\mathcal{O}_\mathfrak{B}, \Gamma)$ implies $\Sf(\mathcal{O}_\mathfrak{B}, \Gamma)$. Conversely, let $X$ satisfy $\Sf(\mathcal{O}_\mathfrak{B}, \Gamma)$. Let $(\mathcal{U}_n)$ be a sequence of open $\mathfrak{B}$-covers of $X$. Without loss generality we assume that for each $n$ $\mathcal{U}_n = \{U_m^n : m\in \omega\}$. We now define a sequence $(\mathcal{W}_n)$ of $\mathcal{O}_\mathfrak{B}$-covers of $X$ by $\mathcal{W}_n = \{\cap_{i\in n+1} U_{m_i}^i : m_i\in \omega\}$. Since $X$ satisfies $\Sf(\mathcal{O}_\mathfrak{B}, \Gamma)$, there exists a sequence $(\mathcal{H}_n)$ such that for each $n$ $\mathcal{H}_n$ is a finite subset of $\mathcal{W}_n$ and $\cup_{n\in \omega} \mathcal{H}_n$ is a $\gamma$-cover of $X$. We can find a sequence $n_0 < n_1 < n_2 < \cdots$ of members of $\omega$ such that for each $k$ $\mathcal{H}_{n_k} \setminus \cup_{j\in k} \mathcal{H}_{n_j} \neq \emptyset$ since $\cup_{n\in \omega} \mathcal{H}_n$ being a $\gamma$-cover of $X$ is infinite. For each $k$ pick $m_0^k, m_1^k, \dotsc, m_{n_k}^k \in \omega$ such that $\cap_{i\in n_k+1} U_{m_i^k}^i \in \mathcal{H}_{n_k} \setminus \cup_{j\in k} \mathcal{H}_{n_j}$. Clearly $\{\cap_{i\in n_k+1} U_{m_i^k}^i : k\in \omega\}$ is a $\gamma$-cover of $X$. Define a sequence $(U_j)$ by
\[U_j = \begin{cases}
          U_{m_j^0}^j, & \mbox{if } j\in n_0+1 \\
          U_{m_j^{k+1}}^j, & \mbox{if } n_k< j \leq n_{k+1} \mbox{ and } k\in \omega.
        \end{cases}\]
Observe that $\{U_n : n\in \omega\}$ is a $\gamma$-cover of $X$. Thus $X$ satisfies $\S1(\mathcal{O}_\mathfrak{B}, \Gamma)$.
\end{proof}

\begin{Th}
\label{thm3}
Let $\bk$ be a bornology on a topological space $X$ with a closed base. Let $X$ be a $\mathfrak{B}$-Lindel\"{o}f space. The following assertions hold.
\begin{enumerate}[wide=0pt, label={\upshape(\arabic*)}, leftmargin=*, ref={\theTh(\arabic*)}]
  \item\label{thm301} $\Sf(\Gamma_\mathfrak{B}, \mathcal{O}) = \Uf(\mathcal{O}_\mathfrak{B}, \mathcal{O})$.

  \item\label{thm302} $\Uf(\Gamma_\mathfrak{B}, \mathcal{O}) = \Uf(\mathcal{O}_\mathfrak{B}, \mathcal{O})$.

  \item\label{thm303} $\Sf(\mathcal{O}_\mathfrak{B}, \mathcal{O}) = \Uf(\mathcal{O}_\mathfrak{B}, \mathcal{O})$.
\end{enumerate}
\end{Th}
\begin{proof}
We detail the proof of (2); the proofs of (1) and (3) are similar and omitted. Clearly, $\Uf(\mathcal{O}_\mathfrak{B}, \mathcal{O})$ implies $\Uf(\Gamma_\mathfrak{B}, \mathcal{O})$. Conversely, suppose that $X$ satisfies $\Uf(\Gamma_\mathfrak{B}, \mathcal{O})$. Let $(\mathcal{U}_n)$ be a sequence of open $\mathfrak{B}$-covers of $X$. Without loss generality we assume that for each $n$ $\mathcal{U}_n = \{U_m^n : m\in \omega\}$. We now define a sequence $(\mathcal{W}_n)$ of $\gamma_\mathfrak{B}$-covers of $X$ by $\mathcal{W}_n = \{V_m^n : m\in \omega\}$ with $V_m^n = \cup_{i\in m+1} U_i^n$. Since $X$ satisfies $\Uf(\Gamma_\mathfrak{B}, \mathcal{O})$, we get a sequence $(m_n : n\in \omega)$ of elements of $\omega$ such that $\{V_{m_n}^n : n\in \omega\}$ covers $X$. For each $n$ choose $\mathcal{V}_n = \{U_i^n : i\in m_n+1\}$. Then for each $n$ $\mathcal{V}_n$ is a finite subset of $\mathcal{U}_n$ and $\{\cup \mathcal{V}_n : n\in \omega\}$ covers $X$. Thus $X$ satisfies $\Uf(\mathcal{O}_\mathfrak{B}, \mathcal{O})$.
\end{proof}

Similarly, the following result can be verified.
\begin{Th}
\label{thm4}
Let $\bk$ be a bornology on a topological space $X$ with a closed base. Let $X$ be a $\mathfrak{B}$-Lindel\"{o}f space. The following assertions hold.
\begin{enumerate}[wide=0pt, label={\upshape(\arabic*)}, leftmargin=*, ref={\theTh(\arabic*)}]
  \item\label{thm401} $\Uf(\mathcal{O}_\mathfrak{B}, \mathcal{O}_\mathfrak{B}) = \Uf(\Gamma_\mathfrak{B}, \mathcal{O}_\mathfrak{B})$.

  \item\label{thm402} $\Uf(\mathcal{O}_\mathfrak{B}, \Gamma_\mathfrak{B}) = \Uf(\Gamma_\mathfrak{B}, \Gamma_\mathfrak{B})$.

  \item\label{thm404} $\Uf(\mathcal{O}_\mathfrak{B}, \Gamma) = \Uf(\Gamma_\mathfrak{B}, \Gamma)$.

  \item\label{thm406} $\Uf(\mathcal{O}_\mathfrak{B}, \Omega) = \Uf(\Gamma_\mathfrak{B}, \Omega)$.
\end{enumerate}
\end{Th}

\begin{Lemma}[{\cite[Lemma 3.1]{tsaban24}}]
\label{lemma6}
Let $\bk$ be a bornology on a topological space $X$ with a closed base. The following assertions are equivalent.
\begin{enumerate}[wide=0pt,label={\upshape(\arabic*)},leftmargin=*]
  \item $X$ satisfies $\Sf(\Gamma_\mathfrak{B}, \Omega)$.
  \item For each $k\in \omega$, $X$ satisfies $\Sf(\Gamma_\mathfrak{B}, \mathcal{O}_k)$.
\end{enumerate}
\end{Lemma}

\begin{Th}
\label{thm7}
Let $\bk$ be a bornology on a topological space $X$ with a closed base. Then $\Uf(\Gamma_\mathfrak{B}, \Gamma)$ implies $\Sf(\Gamma_\mathfrak{B}, \Omega)$.
\end{Th}
\begin{proof}
We closely follow the techniques of the proof of \cite[Theorem 3.2]{tsaban24}. By Lemma~\ref{lemma6}, it is sufficient to show that $X$ satisfies $\Sf(\Gamma_\mathfrak{B}, \mathcal{O}_k)$ for each $k\in \omega$. Let $(\mathcal{U}_n)$ be a sequence of $\gamma_\mathfrak{B}$-covers of $X$ with each $\mathcal{U}_n = \{U_m^n : m\in \omega\}$. Define a sequence $(\mathcal{V}_n)$ of $\gamma_\mathfrak{B}$-covers of $X$ by $\mathcal{V}_n = \{V_m^n : m\in \omega\}$ where $V_m^n = U_m^0\cap U_m^1\cap \cdots \cap U_m^n$. Observe that $V_m^0 \supseteq V_m^1 \supseteq V_m^2 \supseteq \cdots$ for all $m\in \omega$.

Let $g_0(n)=n$ for all $n$. We now define increasing functions $g_1, \dotsc, g_k\in \omega^\omega$ using induction. Let $l<k$ and suppose that $g_l$ is defined. For $n,m,i\in\omega$ let $W_i^{l,n}=\cap_{m=i}^{g_l(i)}V_m^n$ and $H_m^{l,n}=\cup\{W_i^{l,n} : n\leq i,g_l(i)\leq m\}$. For each $l$ and $n$ the set $\{H_m^{l,n} : m\in \omega\}$ is an increasing $\mathcal{O}_\mathfrak{B}$-cover of $X$, i.e. it is an increasing $\gamma_\mathfrak{B}$-cover of $X$. Since $X$ satisfies $\Uf(\Gamma_\mathfrak{B}, \Gamma)$, there exists an increasing $g_{l+1} \in \omega^\omega$ such that $\{H_{g_{l+1}(n)}^{l,n} : n\in \omega\}$ is a $\gamma$-cover of $X$. This completes the inductive construction.

We claim that $\{U_m^n : n\in\omega, m\leq g_k(n)\}$ is a $k$-cover of $X$. Let $F= \{x_i : i\in k \}$ be a $k$-element subset of $X$. Then there exists a $N\in\omega$ such that $F \subseteq H_{g_{l+1}(n)}^{l,n}$ for all $l\in k$ and all $n\geq N$. Fix a $n_0\geq N$.

Since $x_0\in H_{g_k(n_0)}^{k-1,n_0}$, there exists $n_1$ with $n_0\leq n_1, g_{k-1}(n_1) \leq g_k(n_0)$ and $x_0\in W_{n_1}^{k-1,n_0} = \cap_{m=n_1}^{g_{k-1}(n_1)} V_m^{n_0}$.

Since $x_1\in H_{g_{k-1}(n_1)}^{k-2,n_1}$, there exists $n_2$ with $n_1\leq n_2, g_{k-2}(n_2) \leq g_{k-1}(n_1)$ and $x_1\in W_{n_2}^{k-2,n_1} = \cap_{m=n_2}^{g_{k-2}(n_2)} V_m^{n_1} \subseteq \cap_{m=n_2}^{g_{k-2}(n_2)} V_m^{n_0}$.

Since $x_{k-1}\in H_{g_1(n_{k-1})}^{0,n_{k-1}}$, there exists $n_k$ with $n_{k-1} \leq n_k = g_0(n_k) \leq g_1(n_{k-1})$ and $x_{k-1}\in W_{n_k}^{0,n_{k-1}} = V_{n_k}^{n_{k-1}} \subseteq V_{n_k}^{n_0}$.

Thus $F \subseteq V_{n_k}^{n_0} \subseteq U_{n_k}^{n_0}$ and $n_k\leq g_k(n_0)$. It follows that $\{U_m^n : n\in\omega, m\leq g_k(n)\}$ is a $k$-cover of $X$. Hence $X$ satisfies $\Sf(\Gamma_\mathfrak{B}, \mathcal{O}_k)$ for each $k\in\omega$.
\end{proof}

The following implication diagram (Figure~\ref{dig1}) summarizes the relationships established so far.
\begin{figure}[h!]
\begin{adjustbox}{max width=\textwidth, max height=\textheight, keepaspectratio, center}
\begin{tikzcd}[column sep=.55cm, row sep=.7cm, arrows={crossing over}]
&&&&& \Uf(\mathcal{O}_\mathfrak{B}, \mathcal{O}_\mathfrak{B}) \arrow[equal,dddddd] \arrow[rrrrrrd, bend left=10] &&&&&&&&\\
&& \Uf(\mathcal{O}_\mathfrak{B}, \Gamma_\mathfrak{B}) \arrow[equal,dddddd] \arrow[rrru] \arrow[rrrrrr] &&&&&& \Uf(\mathcal{O}_\mathfrak{B}, \Gamma) \arrow[equal,dddddd] \arrow[rrr] &&& \Uf(\mathcal{O}_\mathfrak{B}, \Omega) \arrow[equal,dddddd] \arrow[rrr] &&& \Uf(\mathcal{O}_\mathfrak{B}, \mathcal{O}) \arrow[equal,dddddd]\\
&&&& \Sf(\mathcal{O}_\mathfrak{B}, \mathcal{O}_\mathfrak{B}) \arrow[ruu] \arrow[rrrrrrd, bend left=10] \arrow[dddddd] &&&&&&&&&&\\
&\Sf(\mathcal{O}_\mathfrak{B}, \Gamma_\mathfrak{B}) \arrow[dddddd] \arrow[rrru] \arrow[ruu] \arrow[rrrrrr] &&&&&& \Sf(\mathcal{O}_\mathfrak{B}, \Gamma) \arrow[dddddd] \arrow[ruu] \arrow[rrr] &&& \Sf(\mathcal{O}_\mathfrak{B}, \Omega) \arrow[dddd] \arrow[rrr] \arrow[ruu] &&& \Sf(\mathcal{O}_\mathfrak{B}, \mathcal{O}) \arrow[equal,dddddd] \arrow[equal,ruu] &\\
&&& \S1(\mathcal{O}_\mathfrak{B}, \mathcal{O}_\mathfrak{B}) \arrow[ruu] \arrow[rrrrrrd, bend left=10] \arrow[dddddd] &&&&&&&&&&&\\
\S1(\mathcal{O}_\mathfrak{B}, \Gamma_\mathfrak{B}) \arrow[dddddd] \arrow[rrru] \arrow[equal,ruu] \arrow[rrrrrr] &&&&&& \S1(\mathcal{O}_\mathfrak{B}, \Gamma) \arrow[dddddd] \arrow[equal,ruu] \arrow[rrr] &&& \S1(\mathcal{O}_\mathfrak{B}, \Omega) \arrow[dddddd] \arrow[ruu] \arrow[rrr] &&& \S1(\mathcal{O}_\mathfrak{B}, \mathcal{O}) \arrow[dddddd] \arrow[ruu] &&\\
&&&&&\Uf(\Gamma_\mathfrak{B}, \mathcal{O}_\mathfrak{B}) \arrow[rrrrrrd, bend left=10] &&&&&&&&&\\
&&\Uf(\Gamma_\mathfrak{B}, \Gamma_\mathfrak{B}) \arrow[rrrrrr] \arrow[rrru] &&&&&& \Uf(\Gamma_\mathfrak{B}, \Gamma) \arrow[rr] && \Sf(\Gamma_\mathfrak{B}, \Omega) \arrow[r] & \Uf(\Gamma_\mathfrak{B}, \Omega) \arrow[rrr] &&& \Uf(\Gamma_\mathfrak{B}, \mathcal{O})\\
&&&&\Sf(\Gamma_\mathfrak{B}, \mathcal{O}_\mathfrak{B}) \arrow[ruu] \arrow[rrrrrru, bend right=10] &&&&&&&&&&\\
&\Sf(\Gamma_\mathfrak{B}, \Gamma_\mathfrak{B}) \arrow[ruu] \arrow[rrru] \arrow[rrrrrr] &&&&&& \Sf(\Gamma_\mathfrak{B}, \Gamma) \arrow[rrrrrr] \arrow[ruu] &&&&&& \Sf(\Gamma_\mathfrak{B}, \mathcal{O}) \arrow[equal,ruu] &\\
&&&\S1(\Gamma_\mathfrak{B}, \mathcal{O}_\mathfrak{B}) \arrow[ruu] \arrow[rrrrrrd, bend left=10] &&&&&&&&&&&\\
\S1(\Gamma_\mathfrak{B}, \Gamma_\mathfrak{B}) \arrow[rrrrrr] \arrow[rrru] \arrow[equal,ruu] &&&&&& \S1(\Gamma_\mathfrak{B}, \Gamma) \arrow[equal,ruu] \arrow[rrr] &&& \S1(\Gamma_\mathfrak{B}, \Omega) \arrow[rrr] \arrow[ruuuu] &&& \S1(\Gamma_\mathfrak{B}, \mathcal{O}) \arrow[ruu] &&\\
&&&&&\Uf(\Gamma, \mathcal{O}_\mathfrak{B}) \arrow[equal,ddddd] \arrow[uuuuuu] \arrow[rrrrrrd, bend left=10] &&&&&&&&&\\
&&\Uf(\Gamma, \Gamma_\mathfrak{B}) \arrow[equal,dddd] \arrow[rrrrrr] \arrow[uuuuuu] \arrow[rrru] &&&&&& \Uf(\mathcal{O}, \Gamma) \arrow[rr] \arrow[uuuuuu] && \Sf(\Gamma, \Omega) \arrow[uuuuuu] \arrow[r] & \Uf(\mathcal{O}, \Omega) \arrow[rr] \arrow[uuuuuu] && \Sf(\mathcal{O}, \mathcal{O}) \arrow[uuuu] & \\
&&&&\Sf(\Gamma, \mathcal{O}_\mathfrak{B}) \arrow[ruu] \arrow[rrrrrru, bend right=10] \arrow[uuuuuu] &&&&&&&&&&\\
&\Sf(\Gamma, \Gamma_\mathfrak{B}) \arrow[rrrrr] \arrow[rrru] \arrow[uuuuuu] \arrow[ruu] &&&&& \S1(\Gamma, \Gamma) \arrow[uuuu] \arrow[rruu] \arrow[rrr] &&& \S1(\Gamma, \Omega) \arrow[uuuu] \arrow[ruu] \arrow[rrr] &&& \S1(\Gamma, \mathcal{O}) \arrow[uuuu] \arrow[ruu] &&\\
&&&\S1(\Gamma, \mathcal{O}_\mathfrak{B}) \arrow[ruu] \arrow[rrrrrru, bend right=10] \arrow[uuuuuu] &&&&&&& \Sf(\Omega, \Omega) \arrow[uuu] &&&&\\
\S1(\Gamma, \Gamma_\mathfrak{B}) \arrow[rrru] \arrow[equal,ruu] \arrow[uuuuuu] & \Uf(\mathcal{O}, \Gamma_\mathfrak{B}) \arrow[equal,r] & \Uf(\Omega, \Gamma_\mathfrak{B}) \arrow[rr] && \Uf(\mathcal{O}, \mathcal{O}_\mathfrak{B}) \arrow[equal,r] & \Uf(\Omega, \mathcal{O}_\mathfrak{B}) & \S1(\Omega, \Gamma) \arrow[uu] \arrow[rrr] &&& \S1(\Omega, \Omega) \arrow[ru] \arrow[uu] \arrow[rrr] &&& \S1(\mathcal{O}, \mathcal{O}) \arrow[uu] &&
\end{tikzcd}
\end{adjustbox}
\caption{}
\label{dig1}
\end{figure}

\begin{Ex}
\label{SE1}
Let $X=\omega^\omega$ and $\bk$ be a bornology with a compact base on $X$. We show that $X$ does not satisfies $\Uf(\Gb,\ob)$. Assume the contrary. Let $\uc_n=\{U^n_k:k\in\omega\}$ for each $n$, where $U^n_k=\{f\in \omega^\omega:f(n)\leq k\}$. Clearly, $\{\uc_n:n\in \omega\}$ is a sequence of $\gb$-covers of $X$. Choose a finite subset $\vc_n$ of $\uc_n$ for each $n$ for which $\{\cup \vc_n:n\in \omega\}$ is an open $\bk$-cover of $X$. Define a function $g:\omega\to \omega$ by $g(n)=2\cdot\max\{k\in \nb:U^n_k\in \vc_n\}$. Take a $B\in \bk$ with $g\in B$. We have $B\subseteq \cup\vc_n$ for some $n$ and $g\in U^n_k$ for some $U^n_k\in \vc_n$. Consequently, $g(n)\leq k$ which is a contradiction. Hence  $X$ does not satisfies $\Uf(\Gb,\ob)$.
\end{Ex}

Note that the above bornological space does not satisfy $\Sf(\Gb,\ob)$. If $X$ is an $A_2$-space and $\bk=\mathcal{F}$, then $X$ satisfies $\S1(\Gb,\Gb)$ (see \cite[Theorem 2.10]{coc2}).

\subsection{Some game theoretic observations}
In this section we present certain observations on selection principles related to games.
\begin{Th}[\cite{am}]
\label{TG1}
Let $\bk$ be a bornology on a topological space $X$ with a closed base. The following assertions are equivalent.\\
\noindent$(1)$ $X$ satisfies $\Sf(\ob,\ob)$.\\
\noindent$(2)$ ONE does not have a winning strategy in $\Gf(\ob,\ob)$.
\end{Th}
\begin{Th}[\cite{am}]
\label{TG2}
Let $\bk$ be a bornology on a topological space $X$ with a closed base. The following assertions are equivalent.\\
\noindent$(1)$ $X$ satisfies $\S1(\ob,\ob)$.\\
\noindent$(2)$ ONE does not have a winning strategy in $\G1(\ob,\ob)$.
\end{Th}
\begin{Th}
\label{thm8}
Let $\mathfrak{B}$ be a bornology on $X$. Then the following assertions are equivalent.
\begin{enumerate}[wide=0pt,label={\upshape(\arabic*)},leftmargin=*]
  \item $X$ satisfies $\S1(\Gamma_\mathfrak{B}, \Gamma_\mathfrak{B})$.
  \item ONE does not have a winning strategy in $\G1(\Gamma_\mathfrak{B}, \Gamma_\mathfrak{B})$ on $X$.
\end{enumerate}
\end{Th}
\begin{proof}
$(1) \Rightarrow (2)$. Let $\sigma$ be a strategy for ONE in $\G1(\Gamma_\mathfrak{B}, \Gamma_\mathfrak{B})$ on $X$. Let $\sigma (\emptyset) = \mathcal{U}$ be the first move of ONE. By Lemma~\ref{lemma201}, we can choose $\mathcal{U} = \{U_{(n)} : n\in \omega\}$. TWO responds by selecting a $U_{(n_0)} \in \{U_{(n)} : n\in \omega\}$. Let $\sigma (U_{(n_0)}) \setminus \{U_{(n_0)}\} = \{U_{(n_0, m)} : m\in \omega\}$ be the second move of ONE. TWO responds by choosing a $U_{(n_0, n_1)} \in \{U_{(n_0, m)} : m\in \omega\}$. Let $\sigma (U_{(n_0, n_1)}) \setminus \{U_{(n_0, n_1)}\} = \{U_{(n_0, n_1, m)} : m\in \omega\}$ be the third move of ONE and so on. Thus for each finite sequence $s$ in $\omega$, $\{U_{s \frown (m)} : m\in \omega\} \in \Gamma_\mathfrak{B}$. Since $X$ satisfies $\S1(\Gamma_\mathfrak{B}, \Gamma_\mathfrak{B})$, for each finite sequence $s$ in $\omega$ there exists a $n_s \in \omega$ such that $\{U_{s\frown (n_s)} : s \text{ is a finite sequence in } \omega\} \in \Gamma_\mathfrak{B}$. We now define a sequence $(n_k)$ in $\omega$ as $n_0 = n_\emptyset$, $n_1 = n_{(n_0)}$, $n_2 = n_{(n_0, n_1)}$ and so on. Then $U_{(n_0)}, U_{(n_0, n_1)}, \dotsc$ is the sequence moves of TWO. Since $\{U_{s \frown (m)} : m\in \omega\} \in \Gamma_\mathfrak{B}$, by Lemma~\ref{lemma201}, $\{U_{(n_0)}, U_{(n_0, n_1)}, \dotsc \}$ is also in $\Gamma_\mathfrak{B}$. Thus we get a $\sigma$-play $\{U_{(n)} : n\in \omega\}, U_{(n_0)}, \{U_{(n_0, m)} : m\in \omega\}, U_{(n_0, n_1)}, \{U_{(n_0, n_1, m)} : m\in \omega\}, \dotsc$ which is lost by ONE. It follows that $\sigma$ is not a winning strategy for ONE. Hence ONE does not have a winning strategy in $\G1(\Gamma_\mathfrak{B}, \Gamma_\mathfrak{B})$ on $X$.
\end{proof}

\begin{Th}
\label{thm9}
Let $\bk$ be a bornology on a topological space $X$ with a closed base. The following assertions are equivalent.\\
\noindent$(1)$ $X$ satisfies $\S1(\ob,\Gb)$.\\
\noindent$(2)$ ONE does not have a winning strategy in $\G1(\ob,\Gb)$.
\end{Th}

\begin{Th}
\label{TG3}
Let $\bk$ be a bornology on a topological space $X$ with a closed base. The following assertions are equivalent.\\
\noindent$(1)$ $X$ satisfies $\Sf(\ob,\mathcal{O}^{gp}_{\bk})$.\\
\noindent$(2)$ ONE does not have a winning strategy in $\Gf(\ob,\mathcal{O}^{gp}_{\bk})$.
\end{Th}
\begin{proof}
$(1)\Rightarrow (2)$. Clearly, $X$ satisfies $\S1(\ob,\ob)$. By Theorem \ref{TG1}, ONE does not have a winning strategy in $\Gf(\ob,\ob)$. Also from $(1)$ we obtain that every open $\bk$-cover of $X$ is $\bk$-groupable. Hence ONE does not have a winning strategy in $\Gf(\ob,\mathcal{O}^{gp}_{\bk})$.
\end{proof}
Similarly, we obtain the following.
\begin{Th}
\label{TG4}
Let $\bk$ be a bornology on a topological space $X$ with a closed base. The following assertions are equivalent.\\
\noindent$(1)$ $X$ satisfies $\S1(\ob,\mathcal{O}^{gp}_{\bk})$.\\
\noindent$(2)$ ONE does not have a winning strategy in $\G1(\ob,\mathcal{O}^{gp}_{\bk})$.
\end{Th}

We now present some results on the $\alpha_i$ properties of $X$.

\begin{Th}
\label{Ta1}
Let $\bk$ be a bornology on a topological space $X$ with closed base. The  following assertions are equivalent.\\
\noindent$(1)$ $X$ satisfies $\alpha_2(\Gb,\bfGb)$.\\
\noindent$(2)$ $X$ satisfies $\alpha_3(\Gb,\bfGb)$.\\
\noindent$(3)$ $X$ satisfies $\alpha_4(\Gb,\bfGb)$.\\
\noindent$(4)$ $X$ satisfies $\S1(\Gb,\bfGb)$.\\
\noindent$(5)$ ONE has no winning strategy in $\G1(\Gb,\bfGb)$.\\
\noindent$(6)$ $X$ satisfies $\bigcap_\infty(\Gb,\bfGb)$.
\end{Th}
\begin{proof}
The proofs of the equivalences of $(1)$ to $(4)$ are similar to the proof of \cite[Theorem 3.5]{mk} and $(4)\Leftrightarrow (5)$ follows from  Theorem \ref{thm8}. We now prove $(1)\Leftrightarrow (6)$.

$(1)\Rightarrow (6)$. Let $\{\uc_n:n\in \omega\}$ be a sequence of $\gb$-covers of $X$.
Applying $\alpha_2(\Gb,\bfGb)$, choose a $\wc\in \bfGb$, $\wc=\{W_n:n\in \omega\}$ for which $\uc_n\cap \wc$ is infinite for each $n$. Choose $\wc_n=\uc_n\cap \wc$. We prove that $\{\cap \wc_n:n\in \omega\}\in \bfGb$. We assume the contrary. Then there exists $B_0\in \bk$ for which $B_0\nsubseteq \cap \wc_n$ for infinitely many $n$. Consequently, $B_0\nsubseteq W_n$ for infinitely many $n\in \omega$. This is a contradiction as $\wc\in \bfGb$.
Therefore $\{\cap \wc_n:n\in \omega\}\in \bfGb$. Hence $(6)$ holds.

$(6)\Rightarrow (1)$. Let $\{\uc_n:n\in \omega\}$ be a sequence of $\gb$-covers of $X$. By $(2)$, choose an infinite subset $\wc_n$ of $\uc_n$ for each $n$ for which $\{\cap \wc_n:n\in \omega\}\in \bfGb$. Choose $\wc=\cup_{n\in \omega} \wc_n$. Then for each $n$ $\uc_n\cap \wc$ is infinite.
Enumerate $\wc_n$ as $\{W^n_k:k\in \omega\}$ and so $\wc=\{W^n_k:n,k\in \omega\}$. We prove that $\wc\in \bfGb$. Let $B\in \bk$. Choose $n_0$ for which $B\subseteq \cap \wc_n$ for all $n\geq n_0$. So $B\subseteq W^n_k$ for all $k\in \omega$ and for all $n\geq n_0$. Fix a $n$ with $n\geq n_0$. Then $B\subseteq W^n_k$ for all $k\in \omega$. Note that each $\wc_n$ being an infinite subset of $\uc_n$ is a $\gb$-cover. There exists a finite subset $A_i$ for each $i$, $1\leq i< n_0$ such that $B\subseteq W^i_k$ for all $k\in \omega\setminus A_i$. Choose $A=\cup_{1\leq i<n_0}A_i$. We now have $B\subseteq W^n_k$ for all $k\in \omega\setminus A$ and for all $n\in \omega$. Therefore $\wc\in \bfGb$. Hence $(1)$ holds.
\end{proof}

\begin{Th}
\label{Ta2}
Let $\bk$ be a bornology on a topological space $X$ with closed base. The  following assertions are equivalent.\\
\noindent$(1)$ $X$ satisfies $\alpha_2(\ob,\bfGb)$.\\
\noindent$(2)$ $X$ satisfies $\alpha_3(\ob,\bfGb)$.\\
\noindent$(3)$ $X$ satisfies $\alpha_4(\ob,\bfGb)$.\\
\noindent$(4)$ $X$ satisfies $\S1(\ob,\bfGb)$.\\
\noindent$(5)$ ONE has no winning strategy in $\G1(\ob,\bfGb)$.\\
\noindent$(6)$ $X$ satisfies $\bigcap_\infty(\ob,\bfGb)$.
\end{Th}
\begin{proof}
The proofs of the equivalences of $(1)$ to $(4)$ are similar to the proof of \cite[Theorem 3.3]{mk} and $(4)\Leftrightarrow (5)$ follows from  Theorem \ref{thm9}. We now prove $(4)\Leftrightarrow (6)$.

$(4)\Rightarrow (6)$. Let $\{\uc_n:n\in \omega\}$ be a sequence of open $\bk$-covers of $X$. Since $X$ satisfies $\S1(\ob,\bfGb)$, we know that there is a subset $\vc_n$ of $\uc_n$ for each $n$ which is a $\gb$-cover of $X$ (see \cite[Proposition 5.4]{am}). Note that satisfies $\S1(\Gb,\bfGb)$ and so $X$ also satisfies $\bigcap_\infty(\Gb,\bfGb)$.  Choose a $\wc_n\subseteq \vc_n$ for each $n$ for which $\{\cap \wc_n:n\in \omega\}\in \bfGb$ and also $\wc_n\subseteq \uc_n$. Hence $(6)$ holds.

$(6)\Rightarrow (4)$. Let $\{\uc_n:n\in \omega\}$ be a sequence of open $\bk$-covers of $X$. Choose an infinite subset $\vc_n$ of $\uc_n$ for each $n$ for which $\{\cap \vc_n:n\in \omega\}\in \bfGb$. Again choose a $V_n$ from $\vc_n$ for each $n$. It is easy to see that $\{V_n:n\in \omega\}\in \bfGb$. Hence $(4)$ holds.
\end{proof}

\section{Further observations on selection principles}
\subsection{On product bornologies}
We have the following lemma related an open $\bk^n$-cover of $X^n$, $n\in \omega$.
\begin{Lemma}(\cite[Claim]{am})
\label{L}
Let $\bk$ be a bornology on a topological space $X$ with a compact base $\bk_0$ and $n\in \omega$. Let $\uc$ be an open $\bk^n$-cover of $X^n$. There exists an open $\bk$-cover $\vc$ of $X$ such that $\{V^n:V\in \vc\}$ is a $\bk^n$-cover of $X^n$ which refines $\uc$.
\end{Lemma}
\begin{proof}
Let $\uc$ be an open $\bk^n$-cover of $X^n$. Let $B\in \bk_0$.  We have $B^n\subseteq U$ for some $U\in \uc$. As $B$ is compact, there is an open set $V_B\subseteq X$ for which $B^n\subseteq V^n_B\subseteq U$ (Wallace theorem \cite{Engelking}). Take the family $\vc=\{V_B:B\in \bk_0\}$. It is easy to see that $\vc$ is the open $\bk$-cover satisfying the given condition
\end{proof}

\begin{Prop}
\label{PP-1}
Let $\bk$ be a bornology on a topological space $X$ with a compact base $\bk_0$. The following assertions are equivalent.\\
\noindent$(1)$ $X$ is $\bk$-Lindel\"{o}f.\\
\noindent$(2)$ $X^n$ is $\bk^n$-Lindel\"{o}f for any $n\in \omega$.
\end{Prop}

\begin{Th}
\label{TP-1}
Let $\bk$ be a bornology on a topological space $X$ with a compact base $\bk_0$. Let $\Pi\in \{\S1, \Sf\}$. The following assertions are equivalent.\\
\noindent$(1)$ $X$ satisfies $\Pi(\ob,\ob)$.\\
\noindent$(2)$ $X^n$ satisfies $\Pi(\mathcal{O}_{\bk^n},\mathcal{O}_{\bk^n})$ for any $n\in \omega$.
\end{Th}
\begin{proof}
$(1)\Rightarrow (2)$. We prove only for $\Pi=\S1$. Let $\{\uc_k:k\in \omega\}$ be a sequence of open $\bk^n$-covers of $X^n$. For each $\uc_k$ choose an open $\bk$-cover $\vc_k$ of $X$ for which $\{V^n:V\in \vc_k\}$ is an open $\bk^n$-cover of $X$ that refines $\uc_k$. Applying $\S1(\ob,\ob)$, choose $V_k\in \vc_k$ for each $k$ for which $\{V_k:k\in \omega\}$ is an open $\bk$-cover of $X$. For each $V_k$ choose $U_k\in \uc_k$ satisfying $V_k^n\subseteq U_k$. We prove that $\{U_k:k\in \omega\}$ is an open $\bk^n$-cover of $X^n$. Let $B^n\in \bk^n$, $B\in \bk$. Now $B\subseteq V_k$ for some $k$. Consequently, $B^n\subseteq V^n_k\subseteq U_k$. Hence $\{U_k:k\in \omega\}$ is an open $\bk^n$-cover of $X^n$ witnessing $\S1(\mathcal{O}_{\bk^n},\mathcal{O}_{\bk^n})$. The proof of $(2)\Rightarrow (1)$ is easy.
\end{proof}
Similarly, we obatin the following.
\begin{Th}
\label{TP-2}
Let $\bk$ be a bornology on a topological space $X$ with a compact base $\bk_0$. Let $\Pi\in \{\S1, \Sf\}$. The following statements are equivalent.\\
\noindent$(1)$ $X$ satisfies $\Pi(\ob,\Gb)$.\\
\noindent$(2)$ $X^n$ satisfies $\Pi(\mathcal{O}_{\bk^n},\Gamma_{\bk^n})$ for any $n\in \omega$.
\end{Th}

We state the following lemma which we will use in the next result. The proof is similar to that of \cite[Lemma 11]{cocVIII}.
\begin{Lemma}
\label{Lg}
Let $\bk$ be a bornology on topological space $X$ with a closed base. Let $X$ be $\bk$-Lindel\"{o}f and satisfy $\CDR(\ob,\ob)$. If $X$ satisfies $\Sf(\ob,\mathcal{O}^{gp}_{\bk})$, then for any sequence $\{\uc_n:n\in \omega\}$ of open $\bk$-covers of $X$ there exists a sequence $\{\vc_n:n\in \omega\}$ of pairwise disjoint finite sets with $\vc_n\subseteq \uc_n$ for each $n$ such that for $B\in \bk$ there exists $n_0$ satisfying $B\subseteq V$ for some $V\in \vc_n$ for all $n\geq n_0$.
\end{Lemma}

\begin{Th}
Let $\bk$ be a bornology on a topological space $X$ with a compact base $\bk_0$. Let $X$ be $\bk$-Lindel\"{o}f and satisfy $\CDR(\ob,\ob)$. If $X$ satisfies $\Sf(\ob,\mathcal{O}^{gp}_{\bk})$, then $X^n$ satisfies $\Uf(\mathcal{O}_{\bk^n},\Gamma_{\bk^n})$ for $n\in \omega$.
\end{Th}
\begin{proof}
Fix a $n\in \omega$. Let $\{\uc_k:k\in \omega\}$ be a sequence of open $\bk^n$-covers of $X^n$. By Lemma \ref{L}, choose an open $\bk$-cover $\wc_k$ of $X$ such that $\{W^n:W\in \wc_k\}$ is an open $\bk^n$-cover of $X^n$ which refines $\uc_k$. Applying $\Sf(\ob,\mathcal{O}^{gp}_{\bk})$ and Lemma \ref{Lg}, choose a sequence $\{\vc_k:k\in \omega\}$ of pairwise disjoint finite sets with $\vc_k\subseteq \wc_k$ for each $k$ such that for $B\in \bk$ there is $k_0$ with $B\subseteq V$ for some $V\in \vc_k$ for all $k\geq k_0$. For each $V\in \vc_k$ choose $U_V\in \uc_k$ with $V^n\subseteq U_V$. Set $\zc_k=\{U_V:V^n\subseteq U_V \text{ for } V\in \vc_k\}$. $\zc_k$ is a finite subset of $\uc_k$. Let $B^n\in \bk^n$, $B\in \bk$. Choose $k_0$ with $B\subseteq V$ for some $V\in \vc_k$ for all $k\geq k_0$. Consequently, $B^n\subseteq V^n\subseteq U_V$ for some $U_V\in \zc_k$ for all $k\geq k_0$ and so $B^n\subseteq \cup\zc_k$ for all $k\geq k_0$. Hence $\{\cup\zc_k:k\in \omega\}$ is a $\gamma_{\bk^n}$-cover of $X^n$ witnessing $\Uf(\mathcal{O}_{\bk^n},\Gamma_{\bk^n})$.
\end{proof}

\subsection{Behaviour of selection principles under a continuous function}
The following lemma is easily obtained.
\begin{Lemma}
\label{Lgc1}
Let $\bk$ be a bornology on a topological space $X$ with a closed base and $Y$ be another topological space. Let $g:X\rightarrow Y$ be a continuous function. If $\uc$ is an open $g(\bk)$-cover ($\gamma_{g(\bk)}$-cover) of $g(X)$, then $\{g^{-1}(U):U\in \uc\}$ is an open $\bk$-cover ($\gamma_{\bk}$-cover) of $X$.
\end{Lemma}

\begin{Prop}
\label{Pf1}
Let $\bk$ be a bornology on a topological space $X$ with a closed base $\bk_0$ and $(Y,\rho)$ be a topological space. Let $g:X\rightarrow Y$ be a continuous function and $\Pi\in\{\S1,\Sf,\Uf\}$. The following assertions hold.\\
\noindent$(1)$ If $X$ satisfies $\Pi(\ob, \ob) $, then $f(X)$ satisfies  $\Pi(\ogb,\ogb)$.\\
\noindent$(2)$ If $X$ satisfies $\Pi(\ob, \Gb) $, then $f(X)$ satisfies  $\Pi(\ogb,\Ggb)$.\\
\noindent$(3)$ If $X$ satisfies $\Pi(\Gb, \Gb) $, then $f(X)$ satisfies  $\Pi(\Ggb,\Ggb)$.\\
\noindent$(4)$ If $X$ satisfies $\Pi(\Gb, \ob) $, then $f(X)$ satisfies  $\Pi(\Ggb,\ogb)$.
\end{Prop}
\begin{proof}
We prove only $(2)$ for $\Pi=\S1$ because the remaining proofs are similar. Let  $X$ satisfiy $\S1(\ob,\Gb)$ and let $\{\uc_n:n\in \omega\}$ be a sequence of open $g(\bk)$-covers of $g(X)$. Using Lemma~\ref{Lgc1}, we obtain that for each $n$ $\vc_n=\{g^{-1}(U):U\in \uc_n\}$ is an open $\bk$-cover of $X$. Applying $\S1(\ob,\Gb)$, choose  $g^{-1}(U_n)\in \vc_n$ for each $n$ for which $\{g^{-1}(U_n):n\in \omega\}$ is a $\gb$-cover of $X$. We prove that $\{U_n:n\in \omega\}$ is a $\gamma_{g(\bk)}$-cover of $g(X)$. Take a $B\in g(\bk_0)$ and let $B= g(B^\prime)$ for some $B^\prime\in \bk_0$. Choose $k_0$ for which $B^\prime\subseteq g^{-1}(U_n)$ for all $n\geq  k_0$. Clearly, $B\subseteq U_n$ for all $n\geq  k_0$. Hence $\{U_n:n\in \omega\}$ is a $\gamma_{g(\bk)}$-cover of $g(X)$ witnessing $\S1(\ogb,\Ggb)$.
\end{proof}

\begin{Prop}
\label{Pf3}
Let $\bk$ be a bornology on a topological space $X$ with a compact base $\bk_0$. The following assertions are equivalent.\\
\noindent$(1)$ If $X$ satisfies $\S1(\Gb,\Gb)$, then every continuous image of $X$ into $\omega^\omega$ is bounded.\\
\noindent$(2)$ If $X$ satisfies $\Uf(\Gb,\Gb)$, then every continuous image of $X$ into $\omega^\omega$ is bounded.
\end{Prop}
\begin{proof}
We prove only $(1)$ as the proof of $(2)$ is similar. Let $X$ satisfy $\S1(\Gb,\Gb)$ and $\psi:X\to\omega^\omega$ be a continuous function. By Proposition~\ref{Pf1}, $\psi(X)$ satisfies $\S1(\Gamma_{\psi(\bk)},\Gamma_{\psi(\bk)})$. For $n,k\in \omega$, let $U^n_k=\{f\in \omega^\omega:f(n)\leq k\}$ and $\uc_n=\{U^n_k:k\in \omega\}$ for each $n$. Take a $B\in \psi(\bk_0)$. Since $B$ is compact, there is $k_0\in \omega$ for which $B\subseteq U^n_k$ for all $k\geq k_0$. Therefore $\uc_n$ is a $\gamma_{\psi(\bk)}$-cover of $\psi(X)$. Applying $\S1(\Gamma_{\psi(\bk)},\Gamma_{\psi(\bk)})$, choose  $U^n_{k_n}\in \uc_n$ for each $n$ for which $\{U^n_{k_n}:n\in \omega\}$ is a $\gamma_{\psi(\bk)}$-cover of $\psi(X)$. Define a function $\varphi:\omega\rightarrow \omega$ by $\varphi(n)=k_n$ for each $n\in \omega$. Take a $g\in \varphi(X)$. There  exists $k_0\in \omega$ for which $g\in U^n_{k_n}$ for all $n\geq k_0$. Consequently, $g(n)\leq \varphi(n)$ for all $n\geq k_0$ and so $f\leq^* \varphi$. Hence $\psi(X)$ is bounded in $\omega^\omega$.
\end{proof}

\begin{Prop}
\label{Pf2}
Let $\bk$ be a bornology on a topological space $X$ with a compact base $\bk_0$. The following assertions hold.\\
\noindent$(1)$ If $X$ satisfies $\Sf(\ob,\ob)$, then every continuous image of $X$ into $\omega^\omega$ is not dominating.\\
\noindent$(2)$ If $X$ satisfies $\Uf(\Gb,\ob)$, then every continuous image of $X$ into $\omega^\omega$ is not dominating.
\end{Prop}
\begin{proof}
We prove only $(1)$ because the proof of $(2)$ is similar. Let $X$ satisfy $\Sf(\ob,\ob)$ and $\psi:X\to \omega^\omega$ be a continuous function. By Proposition~\ref{Pf1}, $\psi(X)$ satisfies $\Sf(\oc_{\psi(\bk)},\oc_{\psi(\bk)})$. Consider $\uc_n=\{U^n_k:k\in \omega\}$, where $U^n_k=\{g\in\omega^\omega:g(n)\leq k\}$ for $n,k\in \omega$. Clearly, $\uc_n$ is an open $\bk$-cover of $\psi(X)$. Applying $\Sf(\oc_{\psi(\bk)},\oc_{\psi(\bk)})$, choose a finite subset $\vc_n$ of $\uc_n$ for each $n$ for which $\cup_{n\in \omega}\vc_n$ is an open $\psi(\bk)$-cover of $\psi(X)$. Define a function $\varphi:\omega\rightarrow \omega$ by $\varphi(n)=\max\{k\in \omega:U^n_k\in \vc_n\}$ for each $n\in\omega$. We now prove that for any $f\in \varphi(X)$, $f(n)\leq \varphi(n)$ for infinitely many $n\in \omega$. Let $f\in \psi(X)$. Choose a $B_0\in \varphi(\bk)$ such that $f\in B_0$. Since $\cup_{n\in \omega}\vc_n$ is an open $\psi(\bk)$-cover of $\psi(X)$,  $B_0\subseteq U^n_k$ for some $U^n_k\in \vc_n$ for infinitely many $n\in \omega$. Therefore $f\in U^n_k$ for some $U^n_k\in \vc_n$ for infinitely many $n\in \omega$. Consequently, $f(n)\leq \varphi(n)$ for infinitely many $n\in \omega$. Hence $\psi(X)$ is not dominating.
\end{proof}

We now present some observations on splittability in relation to bornological covers. The relations among $\Split(\mathcal A,\mathcal B)$, where $\mathcal{A},\mathcal{B}\in\{\ob,\Gb\}$ are described in Figure~\ref{splitdiag}. Note that every space satisfies $\Split(\Gb,\Gb)$. Using \cite[Proposition 5.4]{am}, we obtain the following.
\begin{Prop}
\label{Ps1}
Let $\bk$ be a bornology on a topological space $X$ with a closed base. The  following assertions are equivalent.\\
\noindent$(1)$ $X$ is a $\S1(\ob,\Gb)$.\\
\noindent$(2)$ $X$ satisfies $\Split(\ob,\Gb)$.
\end{Prop}

\begin{Prop}
\label{Ps2}
Let $\bk$ be a bornology  on a topological space $X$ with a closed base $\bk_0$ and $Y$ be a topological space. Let $g:X\rightarrow Y$ be a continuous function. The following assertions hold.\\
\noindent$(1)$ If $X$ satisfies $\Split(\ob,\ob)$, then $g(X)$ satisfies $\Split(\ogb,\ogb)$.\\
\noindent$(2)$ If $X$ satisfies $\Split(\ob,\Gb)$, then $g(X)$ satisfies $\Split(\ogb,\Ggb)$.
\end{Prop}
\begin{proof}
We prove only $(1)$ as the proof of $(2)$ is similar. Let $\uc$ be an open $g(\bk)$-cover of $g(X)$. By Lemma~\ref{Lgc1}, $\vc=\{g^{-1}(U):U\in \uc\}$ is an open $\bk$-cover of $X$. By the given condition, choose open $\bk$-subcovers $\vc_1$ and $\vc_2$ of $\vc$ with $\vc_1\cap \vc_2=\emptyset$. Choose $\uc_i=\{U\in \uc:g^{-1}(U)\in \vc_i\}$, $i=1,2$. Clearly, $\uc_1$, $\uc_2$ are open $g(\bk)$-subcovers of $\uc$ with $\uc_1\cap \uc_2=\emptyset$. Hence $g(X)$ satisfies $\Split(\ogb,\ogb)$.
\end{proof}

\begin{figure}
\begin{adjustbox}{keepaspectratio,center}
\begin{tikzcd}[column sep=5ex,row sep=5ex,arrows={crossing over}]
\Split(\ob,\ob)\arrow[r]&\Split(\Gb,\ob)&\\
\Split(\ob,\Gb)\arrow[u]\arrow[r]&\Split(\Gb,\Gb)\arrow[u,equals]&
\end{tikzcd}
\end{adjustbox}
\caption{}
\label{splitdiag}
\end{figure}

\subsection{Certain observations on selection principles related to cardinalities}
\begin{Th}
\label{Tcard1}
Let $\bk$ be a bornology on a topological space $X$ with a closed base $\bk_0$. Let $X$ be a $\bk$-Lindel\"{o}f space. The following statements hold.\\
\noindent$(1)$ If $|\bk_0|<\cov(\mc)$, then $X$ satisfies $\S1(\ob,\ob)$.\\
\noindent$(2)$ If $|\bk_0|<\df$, then $X$ satisfies $\Sf(\ob,\ob)$. \\
\noindent$(3)$ If $|\bk_0|<\pf$, then $X$ satisfies $\S1(\ob,\Gb)$.
\end{Th}
\begin{proof}
$(1)$. Let $\{\uc_n:n\in \omega\}$ be a sequence of open $\bk$-covers of $X$ and let $\uc_n=\{U^n_m:m\in \omega\}$ for each $n\in \omega$. For each $n\in \omega$ and $B\in \bk_0$ $B\subseteq U^n_m$ for some $m\in \omega$. Define a function $h_B:\omega\rightarrow \omega$ by $h_B(n)=\min\{m\in \omega:B\subseteq U^n_m\}$ for each $n\in \omega$. Take the collection $\{h_B:B\in \bk_0\}$. Since $|\bk_0|<\cov(\mc)$, there is an $h:\omega\rightarrow\omega$ such that $\{n\in \omega:h_B(n)=h(n)\}$ is infinite for each $B\in \bk_0$. We prove that $\{U^n_{h(n)}:n\in \omega\}$ is an open $\bk$-cover. Let $B\in \bk$ and choose $B_0\in \bk_0$ with $B \subseteq B_0$. Then $B_0\subseteq U^n_{h_{B_0}(n)}$. Since the set $\{n:h_{B_0}(n)=h(n)\}$ is infinite, choose $k\in \omega$ with $h_{B_0}(k)=h(k)$ and $B \subseteq U^k_{h(k)}$. This shows that $\{U^n_{h(n)}:n\in \omega\}$ is an open $\bk$-cover of $X$ witnessing $\S1(\ob,\ob)$.

$(2)$. Let $\{\uc_n:n\in \omega\}$ be a sequence of open $\bk$-covers of $X$ and let $\uc_n=\{U^n_m:m\in \omega\}$ for each $n$. For each $n\in \omega$ and $B\in\bk_0$ $B\subseteq U^n_m$ for some $m\in \omega$. Define a function $h_B:\omega\rightarrow \omega$ by $h_B(n)=\min\{m\in \omega:B\subseteq U^n_m\}$. Take the collection $\{h_B:B\in \bk_0\}$. As $|\bk_0|<\df$, $\{h_B:B\in \bk_0\}$ is not dominating. Therefore we choose a function $h:\omega\rightarrow \omega$ satisfying $h_B(n)<h(n)$ for infinitely many $n$ and for any $B\in \bk_0$. Set $\vc_n=\{U^n_m:m\le  h(n)\}$ for each $n$. It is easy to see that $\vc_n\subseteq \uc_n$ is a finite for each $n$ and $\cup_{n\in \omega}\vc_n$ is an open $\bk$-cover of $X$ witnessing $\Sf(\ob,\ob)$.

$(3)$. We prove that $X$ is a $\gb$-space. Let $\uc$ be an open $\bk$-cover of $X$ and let $\uc=\{U_n:n\in \omega\}$. For $B\in \bk$ $B\subseteq U_n$ for infinitely many $n$. Set $Q_B=\{n\in \omega:B\subseteq U_n\}$ for each $B\in \bk$.  $Q_B\subseteq \omega$ is infinite. Take the collection $\qc=\{Q_B:B\in \bk_0\}$.  Let $\{Q_{B_1},\dots,Q_{B_k}\}$ be a finite collection from $\qc$. As $B_1\cup\dots \cup B_k\in \bk$, we have $B_1\cup\dots \cup B_k\subseteq U_n$ for infinitely many $n$. Consequently, $Q_{B_1}\cap\dotsc \cap Q_{B_k}$ is infinite. Thus any finite subfamily of $\qc$ has infinite intersection. As $|\bk_0|<\pf$, choose an infinite subset $P$ of $\omega$ for which $P\setminus Q_B$ is finite for each $B\in \bk_0$. Let $P=\{n_k:k\in \omega\}$ and take $\wc=\{U_{n_k}\in \uc:k\in \omega\}$. We now prove that $\wc$ is a $\gb$-cover. For $B\in \bk_0$ we have that $P\setminus Q_B$ is finite and so we find $k_0\in \omega$ satisfying $B\subseteq U_{n_k}$ for all $k\geq k_0$. Therefore  $\wc\subseteq \uc$ is a $\gb$-cover of $X$. Hence $X$ is a $\gb$-space and from \cite[Proposition 5.4]{am} we obtain that $X$ satisfies $\S1(\ob,\Gb)$.
\end{proof}

\begin{Th}
\label{Tcard2}
Let $\bk$ be a bornology on a topological space $X$ with a closed base $\bk_0$. The following assertions hold.\\
\noindent$(1)$ If $|\bk_0|<\df$, then $X$ satisfies $\S1(\Gb,\ob)$.\\
\noindent$(2)$ If $|\bk_0|<\df$, then $X$ satisfies $\Uf(\Gb,\ob)$.
\end{Th}
\begin{proof}
$(1)$. Let $\{\uc_n:n\in \omega\}$ be a sequence of $\gb$-covers of $X$, where $\uc_n=\{U^n_m:m\in \omega\}$ for each $n$. For each $B\in\bk_0$ there exists $m_0\in \omega$ satisfying $B\subseteq U^n_m$ for all $m\geq m_0$. Define a function $h_B:\omega\rightarrow \omega$ by $h_B(n)=\min\{k\in\omega: \text{ for all } m\ge k,B\subseteq U^n_m\}$. Take the collection $\{h_B:B\in \bk_0\}$. As $|\bk_0|<\df$, $\{h_B:B\in \bk_0\}$ is not dominating. Therefore we choose a function $h:\omega\rightarrow \omega$ satisfying $h_B(n)<h(n)$ for infinitely many $n$ and for any $B\in \bk_0$. It is easy to see that $\{U^n_{h(n)}:n\in \omega\}$ is an open $\bk$-cover of $X$ witnessing $\S1(\Gb,\ob)$. The proof of $(2)$ is similar.
\end{proof}

\begin{Th}
\label{Tcard2}
Let $\bk$ be a bornology on a topological space $X$ with a closed base $\bk_0$. The following assertions hold.\\
\noindent$(1)$ If $|\bk_0|<\bb$, then $X$ satisfies $\S1(\Gb,\Gb)$.\\
\noindent$(2)$ If $|\bk_0|<\bb$, then $X$ satisfies $\Uf(\Gb,\Gb)$.\\
\noindent$(3)$ If $|\bk_0|<\bb$, then for any sequence $\{\uc_n:n\in \omega\}$ of open $\bk$-covers of $X$ there is a finite subset $\vc_n$ of $\uc_n$ such that for a $B\in \bk$ there exists $n_0\in \omega$ for which $B\subseteq V$ for some $V\in \vc_n$ for all $n\geq n_0$.
\end{Th}
\begin{proof}
$(1)$. Let $\{\uc_n:n\in \omega\}$ be a sequence of $\gb$-covers of $X$, where $\uc_n=\{U^n_m:m\in \omega\}$ for each $n$. For each $n\in \omega$ and $B\in\bk_0$ there exists $m_0\in \omega$ satisfying $B\subseteq U^n_m$ for all $m\geq m_0$. Define a function $h_B:\omega\rightarrow \omega$ by $h_B(n)=\min\{k\in \omega: \text{ for all } m\ge k,B\subseteq U^n_m\}$. Take the collection $\{h_B:B\in \bk_0\}$. As $|\bk_0|<\bb$, choose a function $h:\omega\rightarrow \omega$ satisfying $h_B\leq ^*h$ for all $B\in \bk_0$. It is easy to see that $\{U^n_{h(n)}:n\in \omega\}$ is a $\gb$-cover of $X$ witnessing $\S1(\Gb,\Gb)$.

We show $(3)$ as the proof of $(2)$ is similar. Let $\{\uc_n:n\in \omega\}$ be a sequence of open $\bk$-covers of $X$ and let $\uc_n=\{U^n_m:m\in \omega\}$. For each $n\in \omega$ and $B\in\bk_0$ $B\subseteq U^n_m$ for some $m\in \omega$. Define a function $h_B:\omega\to \omega$ by $h_B(n)=\min\{m\in \omega:B\subseteq U^n_m\}$ for each $B\in \bk_0$. Taking the collection $\{h_B:B\in \bk_0\}$ and applying $|\bk_0|<\bb$, choose $h:\omega\to \omega$ with $h_B\leq^* h$ for all $B\in \bk_0$. Set $\vc_n=\{U^n_k:k\leq f(n)\}$ for each $n$. Then for a $B\in \bk_0$ there is $n_0\in \omega$ with $h_B(n)\leq h(n)$ for all $n\geq n_0$. Consequently, $B\subseteq U^n_k$ for some $U^n_k\in \vc_n$ for all $n\geq n_0$.
\end{proof}

\begin{Prop}
\label{Ts3}
Let $\bk$ be a bornology on a topological space $X$ with a closed base $\bk_0$ and let $X$ be a $\bk$-Lindel\"{o}f space. If $|\bk_0|<\rf$, then $X$ satisfies $\Split(\ob,\ob)$.
\end{Prop}
\begin{proof}
Let $\uc$ be an open $\bk$-cover of $X$ and let $\uc=\{U_n:n\in \omega\}$. For $B\in \bk_0$ $B\subseteq U_n$ for infinitely many $n\in \omega$. Let $Q_B=\{n\in \omega:B\subseteq U_n\}$. Clearly, $Q_B$ is infinite. Take the collection $\{Q_B:B\in \bk_0\}$. As $|\bk_0|<\rf$, there exists an infinite subset $Q$ of $\omega$ satisfying $Q_B\nsubseteq^* Q$ and $Q_B\nsubseteq^* \omega\setminus Q$ for every $B\in \bk_0$. $Q_B\setminus Q$ and $Q_B\cap Q$ are infinite subsets. Now $\{U_n:n\in Q\}$ and $\{U_n:n\not\in Q\}$ are disjoint subsets of $\uc$ and are open $\bk$-covers of $X$. This completes the proof.
\end{proof}

From Theorems \ref{Ta1}, \ref{Ta2} and Theorems \ref{Tcard1}(3), \ref{Tcard2}(1) we obtain the following corollary.
\begin{Cor}
Let $\bk$ be a bornology on a topological space $X$ with a closed base $\bk_0$. The following assertions hold.\\
\noindent$(1)$ If $|\bk_0|<\pf$, then $X$ satisfies $\bigcap_\infty(\ob,\bfGb)$, provided $X$ is $\bk$-Lindel\"{o}f.\\
\noindent$(2)$ If $|\bk_0|<\bb$, then $X$ satisfies $\bigcap_\infty(\Gb,\bfGb)$.
\end{Cor}

\section{Certain observations on $C_{\bk}(X)$ space}
In this section we focus on the function space $C_{\bk}(X)$ endowed with the topology $\tau_{\bk}$ of uniform convergence on $\bk$. We investigate certain topological properties of the space $C_{\bk}(X)$ and their relations with selective bornological covering properties of $X$. In this connection we follow the study of these topological properties in the $C_p(X)$ space and the $C_k(X)$ space endowed with the topology of pointwise convergence and the compact-open topology, respectively. Note that $C_{\bk}(X)$ is homogeneous so while studying local properties of $C_{\bk}(X)$ it is enough to focus on the zero function $\underline{0}$. We first present the following lemma from \cite{am} which we will use frequently.

\begin{Lemma}(\cite[Lemma 3.6]{am})
\label{L1}
Let $\bk$ be a bornology on a Tychonoff space $X$ with a compact base. The following assertions hold.\\
\noindent$(1)$ Let $\uc$ be a family of open subsets of $X$ with $X\not\in \uc$ and let $A=\{f\in C_{\bk}(X): f(X\setminus U)=\{1\} \text{ for some } U\in \uc\}$. $\uc$ is an open $\bk$-cover of $X$ if and only if $A\in \Omega_{\underline{0}}$.\\
\noindent$(2)$ Let $A$ be a subset of $C_{\bk}(X)$ and let $\uc_n(A)=\{f^{-1}(-\frac{1}{n},\frac{1}{n}):f\in A\}$ for $n\in \omega$. If $\underline{0}\in \bar{A}$, then $\uc_n(A)$ is an open $\bk$-cover of $X$.
\end{Lemma}

\begin{Lemma}
\label{L3}
Let $\bk$ be a bornology on a topological space $X$ with a closed base. Let $\{f_n:n\in \omega\}$ be a sequence in $C_{\bk}(X)$ converge to the zero function $\underline{0}$, then for any $k\in \omega$ the sequence $\{f^{-1}_n(-\frac{1}{k},\frac{1}{k}):n\in \omega\}$, where $X\neq f^{-1}_n(-\frac{1}{k},\frac{1}{k})$ for any $n$, is a $\gb$-cover of $X$.
\end{Lemma}
\begin{proof}
Fix a $k\in \omega$. Let $B\in \bk$ and take the neighbourhood $[B,\frac{1}{k}](\underline{0})$. Choose $n_0$ for which $f_n\in [B,\frac{1}{k}](\underline{0})$ for all $n\geq n_0$. Consequently, $B\subseteq f^{-1}_n(-\frac{1}{k},\frac{1}{k})$ for all $n\geq n_0$. Hence $\{f^{-1}_n(-\frac{1}{k},\frac{1}{k}):n\in \omega\}$ is a $\gb$-cover of $X$.
\end{proof}

\begin{Lemma}
\label{L2}
Let $\bk$ be a bornology on a Tychonoff space $X$ with a compact base. Let $\uc$ be a family of open subsets of $X$. $\uc$ is an open $\bk$-cover of $X$ if and only if for each $U$ there exists a closed subset $C(U)$ of $U$ for which $\{C(U):U\in \uc\}$ is a $\bk$-cover of $X$.
\end{Lemma}
\begin{proof}
Suppose that $\uc$ is an open $\bk$-cover of $X$ and let $A=\{f\in C_{\bk}(X): f(X\setminus U)=\{1\}  \text{ for some } U\in \uc\}$. Then $A\in \Omega_{\underline{0}}$ by Lemma \ref{L1}. For $U\in \uc$ let $C(U)=f^{-1}_U([-\frac{1}{2},\frac{1}{2}])$, where $f_U(X\setminus U)=\{1\}$ for $f_U\in A$. Clearly, $C(U)\subseteq U$. Now for a $B\in \bk$ and the neighbourhood $[B,\frac{1}{2}](\underline{0})$ choose $f_U\in [B,\frac{1}{2}](\underline{0})\cap A$ for some $U\in \uc$. Consequently, $B\subseteq C(U)$. This proves that $\{C(U):U\in \uc\}$ is a $\bk$-cover of $X$.  The converse is straightforward.
\end{proof}

\begin{Prop}
\label{Pfun1}
Let $\bk$ be a bornology with closed base on a metric space $X$. Let $\Pi\in \{\S1,\Sf\}$. The following assertions hold.\\
\noindent$(1)$ If $X$ satisfies $\Pi(\Gb, \ob)$, then  $C_{\bk}(X)$ satisfies $\Pi(\Sigma_{\underline{0}},\Omega_{\underline{0}})$.\\
\noindent$(2)$ If $X$ satisfies $\Pi(\Gb, \Gb)$, then  $C_{\bk}(X)$ satisfies $\Pi(\Sigma_{\underline{0}},\Sigma_{\underline{0}})$.
\end{Prop}
\begin{proof}
We prove only $(2)$ for $\Pi=\S1$ as the proof of $(1)$ and other cases are similar. Suppose that $\{A_n:n\in \omega\}$ is a sequence from $\Sigma_{\underline{0}}$, where $A_n=\{f_{n,k}:k\in \omega\}$ for $n\in \omega$. Let $\uc_n=\{f^{-1}_{n,k}(-\frac{1}{n},\frac{1}{n}):k\in \omega\}$. If $X\in \uc_n$ for infinitely many $n$, then choose an  infinite set $M\subseteq \omega$ with $X=f^{-1}_{n,k_n}(-\frac{1}{n},\frac{1}{n})$ for all $n\in M$. Clearly, $\{f_{n,k_n}:n\in M\}$ converges to $\underline{0}$. Now by Lemma \ref{L3}, $\uc_n$, $n\not\in M$, is a $\gb$-cover of $X$. Applying $\S1(\Gb, \Gb)$ to the sequence $\{\uc_n:n\in \omega, n\not\in M\}$, choose $f_{n,k_n}$ for which $\{f^{-1}_{n,k_n}(-\frac{1}{n},\frac{1}{n}):n\in \omega, n\not\in M\}$ is a $\gb$-cover of $X$.
We show that $\{f_{n,k_n}:n\in \omega, n\not\in M\}$ converges to $\underline{0}$. Let $[B,\varepsilon](\underline{0})$ be a neighbourhood of $\underline{0}$, where $B\in \bk$ and $\varepsilon>0$. Choose $n_0\in \omega$ with $\frac{1}{n_0}<\varepsilon$ for which $B\subseteq f^{-1}_{n,k_n}(-\frac{1}{n},\frac{1}{n}))$ for all $n\geq n_0$, $n\not\in M$. Consequently, $f_{n,k_n}\in [B,\varepsilon](\underline{0})$ for all $n\geq n_0$, $n\not\in M$. Hence the sequence $\{f_{n,k_n}:n\in \omega\}$ converges to $\underline{0}$ witnessing  $\S1(\Sigma_{\underline{0}},\Sigma_{\underline{0}})$.
\end{proof}

We now present a characterization of the Reznichenko property of $C_{\bk}(X)$ in terms of a $\bk$-covering property of $X$. The proof is obtained by using Lemma \ref{L2} and by following the line of arguments of the proof of \cite[Theorem 12]{sakai} which is for the $C_p(X)$ space.
\begin{Th}
\label{Tf1}
Let $\bk$ be a bornology on a Tychonoff space $X$ with a compact base. The following assertions are equivalent.\\
\noindent$(1)$ $C_{\bk}(X)$ has the Reznichenko property.\\
\noindent$(2)$ Every open $\bk$-cover of $X$ is $\bk$-groupable.
\end{Th}

The proof of the following game theoretic observations for the space $C_{\bk}(X)$ is adaptation of the proof of \cite[Theorem 4.1]{kocinac03} which is for the space $C_k(X)$.

\begin{Th}
\label{Tf2}
Let $\bk$ be a bornology on a Tychonoff space $X$ with a compact base. The following assertions hold.\\
\noindent$(1)$ If ONE has no wining strategy in the game $\G1 (\ob,\mathcal{O}^{gp}_{\bk})$ on $X$, then $C_{\bk}(X)$ satisfies $\S1(\Omega_{\underline{0}},\Omega_{\underline{0}}^{gp})$.\\
\noindent$(2)$ If ONE has no wining strategy in the game $\Gf(\ob,\mathcal{O}^{gp}_{\bk})$ on $X$, then $C_{\bk}(X)$ satisfies $\Sf(\Omega_{\underline{0}},\Omega_{\underline{0}}^{gp})$.
\end{Th}

\begin{Th}
\label{Tf3}
Let $\bk$ be  bornology on a Tychonoff space $X$ with a compact base. The following assertions are equivalent.\\
\noindent$(1)$ $X$ satisfies $\S1(\ob,\mathcal{O}^{gp}_{\bk})$.\\
\noindent$(2)$ $C_{\bk}(X)$ has countable strong fan tightness and Reznichenko's property.
\end{Th}
\begin{proof}
$(1)\Rightarrow (2)$. By Theorem \ref{TG4}, ONE has no wining strategy in the game $\G1 (\ob,\mathcal{O}^{gp}_{\bk})$ on $X$. Again by Theorem \ref{Tf2}(1) $C_{\bk}(X)$ satisfies $\S1(\Omega_{\underline{0}},\Omega_{\underline{0}}^{gp})$, i.e., $C_{\bk}(X)$ has countable strong fan tightness and Reznichenko's property.

$(2)\Rightarrow (1)$. If $C_{\bk}(X)$ has countable strong fan tightness, then $X$ satisfies $\S1(\ob,\ob)$ (see \cite[Theorem 3.5]{am}). By Theorem \ref{Tf1}, every open $\bk$-cover is $\bk$-groupable. Hence $X$ satisfies $\S1(\ob,\mathcal{O}^{gp}_{\bk})$.
\end{proof}

Similarly, we obtain the following.

\begin{Th}
\label{Tf4}
Let $\bk$ be  bornology on a Tychonoff space $X$ with a compact base. The following assertions are equivalent.\\
\noindent$(1)$ $X$ satisfies $\Sf(\ob,\mathcal{O}^{gp}_{\bk})$.\\
\noindent$(2)$ $C_{\bk}(X)$ has countable fan tightness and Reznichenko's property.
\end{Th}

\subsection{Results on variations on tightness}
In \cite[Theorem 3.1]{kocinac03} the countable $T$-tightness property of $C_k(X)$ space was characeterized in terms of a $k$-covering property of $X$. Accordingly, a similar characterization for the $C_{\bk}(X)$ space in terms of a $\bk$-covering property of $X$ is obtained. We present the result without proof as it is similar to that of \cite[Theorem 3.1]{kocinac03}.
\begin{Th}
\label{Tf5}
Let $\bk$ a bornology on a Tychonoff space $X$ with a compact base. The following assertions are equivalent.\\
\noindent$(1)$ $C_{\bk}(X)$ has countable $T$-tightness.\\
\noindent$(2)$ For each uncountable regular cardinal $\rho$ and each increasing sequence $\{\uc_{\alpha}:\alpha<\rho\}$ of families of open subsets of $X$ for which $\bigcup_{\alpha<\rho}\uc_{\alpha}$ is an open $\bk$-cover of $X$, there is a $\beta<\rho$ with $\uc_{\beta}$ being an open $\bk$-cover of $X$.
\end{Th}

We know that $C_p(X)$ has countable fan tightness for finite sets if and only if every finite power of $X$ is Menger (see  \cite[Lemma 2.8]{sakai2}). Also, $C_{\bk}(X)$ has countable fan tightness (countable strong fan tightness) if and only if $X$ satisfies $\Sf(\ob,\ob)$ ($\S1(\ob,\ob)$) (see \cite[Theorem 3.5]{am}). We now have the following result.

\begin{Th}
\label{Tf6}
Let $\bk$ be a bornology on a Tychonoff space $X$ with a compact base. The  following assertions are equivalent.\\
\noindent$(1)$ $X$ satisfies $\S1(\ob,\ob)$.\\
\noindent$(2)$ $C_{\bk}(X)$ has countable strong fan tightness.\\
\noindent$(3)$ $C_{\bk}(X)$ has countable strong fan tightness for finite sets.
\end{Th}
\begin{proof}
It is enough to prove $(1)\Rightarrow (3)$. Suppose that $\{\ac_n:n\in \omega\}$ is a sequence of $\pi$-networks at $\underline{0}$ consisting of finite subsets of $C_{\bk}(X)$. Choose $U(A)=\cap_{f\in A}f^{-1}(-\frac{1}{n},\frac{1}{n})$ for each $A\in \ac_n$. Take $\uc_n=\{U(A):A\in \ac_n\}$. If $X\in \uc_{n}$ for infinitely many $n$, then choose an infinite subset $M$ of $\omega$ with $X=U(A_{n})$ for all $n\in M$. Clearly, $\{A_{n}:n\in M\}$ is a $\pi$-network at $\underline{0}$ and we are done. We therefore assume that $X\not\in \uc_n$ for each $n\in \omega$. It is easy to see that $\uc_n$ is an open $\bk$-cover of $X$. Using $(1)$, choose $U(A_n)\in \uc_n$ for each $n$ for which $\{U(A_n):n\in \omega\}$ is an open $\bk$-cover of $X$. We prove that $\{A_n:n\in \omega\}$ is a $\pi$-network at $\underline{0}$. Take a neighbourhood $[B,\varepsilon](\underline{0})$ of $\underline{0}$, where $B\in \bk$ and $\varepsilon>0$.  Choose a $n\in \omega$ with $\frac{1}{n}<\varepsilon$ and $U(A_n)$ satisfying $B\subseteq U(A_n)$. Consequently, $B\subseteq f^{-1}(-\varepsilon,\varepsilon)$ for all $f\in A_n$ and so $A_n\subseteq [B,\varepsilon](\underline{0})$. Hence $\{A_n:n\in \omega\}$ is a $\pi$-network at $\underline{0}$.
\end{proof}

Similarly, we obtain the following.
\begin{Th}
\label{Tf6}
Let $\bk$ be a bornology on a Tychonoff space $X$ with a compact base. The  following assertions are equivalent.\\
\noindent$(1)$ $X$ satisfies $\Sf(\ob,\ob)$.\\
\noindent$(2)$ $C_{\bk}(X)$ has countable fan tightness.\\
\noindent$(3)$ $C_{\bk}(X)$ has countable fan tightness for finite sets.
\end{Th}

The space $C_p(X)$ is Fr\'{e}chet-Urysohn for finite sets if and only if $C_p(X)$ is Fr\'{e}chet-Urysohn (see \cite[Page 3]{gs}). In the context of $C_{\bk}(X)$ space we have the observation that $C_{\bk}(X)$ is Fr\'{e}chet-Urysohn is equivalent to the fact that $X$ satisfies $\S1(\ob,\Gb)$ (also $X$ is a $\gb$-space) (see \cite[Proposition 5.4]{am}). The following can now be established.
\begin{Th}
\label{Tf7}
Let $\bk$ be a bornology on a Tychonoff space $X$ with a compact base. The following assertions are equivalent.\\
\noindent$(1)$ $X$ satisfies $\S1(\ob,\Gb)$. \\
\noindent$(2)$ $X$ is a $\gb$-space. \\
\noindent$(3)$ $C_{\bk}(X)$ is strictly Fr\'{e}chet-Urysohn.\\
\noindent$(4)$ $C_{\bk}(X)$ is Fr\'{e}chet-Urysohn.\\
\noindent$(5)$ $C_{\bk}(X)$ is Fr\'{e}chet-Urysohn for finite sets.
\end{Th}

Next, we discuss the supertightness of $C_{\bk}(X)$. We have the following result on the supertightness of $C_p(X)$ from \cite{sakai-st}.

\begin{Th}(\cite[Theorem 2.3]{sakai-st})
\label{Tf10}
For a Tychonoff space $X$ $st(C_p(X))=t(C_p(X))$ for each $n\in \omega$.
\end{Th}
Translating this result in $C_{\bk}(X)$ we have a similar observation.

\begin{Th}
\label{Tf9}
Let $\bk$ be a bornology on a Tychonoff space $X$ with a compact base. Then $st(C_{\bk}(X))=t(C_{\bk}(X))$.
\end{Th}
\begin{proof}
The inequality $st(C_{\bk}(X))\geq t(C_{\bk}(X))$ is straightforward and so we prove only $st(C_{\bk}(X))\leq t(C_{\bk}(X))$. Suppose that $t(C_{\bk}(X))=\mf$ and $\ac$ is a $\pi$-network at $\underline{0}$ consisting of finite subsets of $C_{\bk}(X)$.  Define $U_n(A)=\cap_{f\in A}f^{-1}(-\frac{1}{n},\frac{1}{n})$ for each $A\in \ac$ and $n\in \omega$. Take $F_n=\{f\in C_{\bk}(X): f(X\setminus U_n(A))=\{1\}, A\in \ac\}$. It is easy to see that $\underline{0}\in \overline{F_n}$ for each $n\in \omega$. Since $t(C_{\bk}(X))=\mf$, choose $G_n\subseteq F_n$ for each $n$ with $|G_n|\leq\mf$ for which $\underline{0}\in \overline{G_n}$. Now for each $g\in G_n$ we have a set $A_g$ in $\ac$ with $g(X\setminus U_n(A_g))=\{1\}$. Take the collection $\ac_n=\{A_g\in \ac:g\in G_n\}$. As $|G_n|\leq\mf$, $|\ac_n|\leq\mf$ for each $n$. Let $\gc=\cup_{n\in \omega}\ac_n$. We prove that $\gc$ is a $\pi$-network at $\underline{0}$. For this take a neighbourhood $[B,\varepsilon](\underline{0})$ of $\underline{0}$, where $B\in \bk$ and $\varepsilon>0$. Choose a $n\in \omega$ with $\frac{1}{n}<\varepsilon$. Since $\underline{0}\in \overline{G_n}$. We have a $g\in [B,\frac{1}{n}](\underline{0})\cap G_n$. Consequently, $B\subseteq g^{-1}(-\frac{1}{n},\frac{1}{n})$ and $g(X\setminus U_n(A_g))=\{1\}$. Therefore $B \subseteq U_n(A_g)$ and so $B\subseteq \cap_{f\in A_g}f^{-1}(-\frac{1}{n},\frac{1}{n})$. Clearly, $f\in [B,\frac{1}{n}](\underline{0})$ for all $f\in A_g$ and $A_g\subseteq [B,\varepsilon](\underline{0})$. This shows that $\gc$ is a $\pi$-network at $\underline{0}$ of cardinality less than or equal to $\mf$. Hence $st(C_{\bk}(X))\leq t(C_{\bk}(X))$.
\end{proof}

\begin{Cor}
\label{Ct1}
Let $\bk$ be a bornology on a Tychonoff space $X$ with a compact base. The following assertions are equivalent.\\
\noindent$(1)$ $C_{\bk}(X)$  has countable tightness. \\
\noindent$(2)$ $C_{\bk}(X)$  has countable supertightness.
\end{Cor}

\begin{Th}(\cite[Theorem 2.1]{sakai-st})
\label{Tf10}
For a Tychonoff space $X$ $l(C_p(X))\geq st(X^n)$ for each $n\in \omega$.
\end{Th}
In $C_{\bk}(X)$ we establish a similar result. The proof follows essentially the same line of arguments as that of \cite[Theorem 2.1]{sakai-st} with certain necessary modifications. For the sake of completeness we present the detailed proof.

\begin{Th}
\label{Tf11}
Let $\bk$ be a bornology on a Tychonoff space $X$ with a compact base. Then $l(C_{\bk}(X))\geq st(X^n)$ for each $n\in\omega$.
\end{Th}
\begin{proof}
Let $l(C_{\bk}(X))=\mf$. Fix a $n\in \omega$ and $(x_1,\dots,x_n)\in X^n$. Let $\ac$ be a $\pi$-network at $(x_1,\dots,x_n)$ consisting of finite subsets of $X^n$. We show that there exists a subfamily $\dc$ of $\ac$ with $|\dc|\leq\mf$ which is a $\pi$-network at $(x_1,\dots,x_n)$. Choose a neighbourhood $V_i$ of $x_i$ for each $i=1,\dots,n$ such that $V_i\cap V_j=\emptyset$ when $x_i\neq x_j$ and $V_i=V_j$ when $x_i=x_j$. We assume that $A\subseteq V_1\times\cdots\times V_n$ for all $A\in \ac$. Let $F=\{x_1,\dots,x_n\}$. Consider $C_{\bk}(X;F)=\{f\in C_{\bk}(X):f(F)=\{0\}\}$. Clearly, $C_{\bk}(X;F)$ is closed and so $l(C_{\bk}(X;F))\leq \mf$. We choose a finite set $B_A\in \bk$ for each $A\in \ac$ such that $A=\{(y_1,\dots,y_n):y_1,\dots, y_n\in B_A\}$ ($B_A$ contains only those elements which appear in the representation of the elements of $A$). We prove that $C_{\bk}(X;F)\subseteq \cup_{A\in \ac}[B_A,1](\underline{0})$. Take a $f\in C_{\bk}(X;F)$. As $f(x_i)=0$, $i=1,\dots,n$, $f^{-1}(-1,1)\times \cdots \times f^{-1}(-1,1)$ is a neighbourhood of $(x_1,\dots,x_n)$. Choose an $A\in \ac$ for which $A\subseteq f^{-1}(-1,1)\times \cdots \times f^{-1}(-1,1)$. That is, $B_A\subseteq  f^{-1}(-1,1)$ and so $f\in [B_A,1](\underline{0})$. This proves that $C_{\bk}(X;F)\subseteq \cup_{A\in \ac}[B_A,1](\underline{0})$. Therefore there exists a subfamily $\dc$ of $\ac$ with $|\dc|\leq\mf$ and $C_{\bk}(X;F)\subseteq \cup_{A\in \dc}[B_A,1](\underline{0})$. We claim that $\dc$ is a $\pi$-network at $(x_1,\dots,x_n)$. Let $W_i$ be an open neighbourhood of $x_i$, $i=1,\dots,n$. We assume that $W_i\subseteq V_i$ and $W_i=W_j$ when $x_i=x_j$. Choose a $f\in  C_{\bk}(X)$ with $f(A)=\{0\}$ and $f(X\setminus W_1\cup\cdots \cup W_n)=\{1\}$. Clearly, $f\in C_{\bk}(X;F)$. Now there exists a $A\in \ac$ with $f\in [B_A,1](\underline{0})$. Consequently, $B_A\subseteq f^{-1}(-1,1)$ and so $B_A\subseteq W_1\cup\cdots W_n$. Now for each $(y_1,\dots, y_n)\in A$ we obtain that $(y_1,\dots,y_n)\in W_{i_1}\times \cdots \times W_{i_n}\cap V_1\times\cdots\times V_n$. Clearly, $W_{i_k}\cap V_k\neq \emptyset$ for each $k=1,\dots,n$. As $W_k\subset V_k$ and $W_k$'s as well as $V_k$'s are disjoint, we have $W_{i_k}=W_k$ for each $k=1,\dots,n$. Therefore $A\subseteq W_1\times\cdots \times W_n$. Hence $\dc$ is a $\pi$-network at $(x_1,\dots,x_n)$ and so $l(C_{\bk}(X))\geq st(X^n)$.
\end{proof}

\begin{Cor}
\label{Ct2}
Let $\bk$ be a bornology on a Tychonoff space $X$ with a compact base. If $C_{\bk}(X)$ is Lindel\"{o}f, then $X^n$ has countable supertightness for all $n\in \omega$.
\end{Cor}

\end{document}